\documentclass[12pt]{amsart}
\usepackage{amssymb,amsmath,amsthm}
\usepackage[all,cmtip]{xy}
\usepackage{graphicx}
\usepackage{subcaption}
\newtheorem{theorem}{Theorem}
\newtheorem{lemma}[theorem]{Lemma}
\newtheorem{remark}[theorem]{Remark}
\newtheorem{corollary}[theorem]{Corollary}
\newtheorem{proposition}[theorem]{Proposition}
\numberwithin{theorem}{section} \numberwithin{equation}{section}
\newcommand\scalemath[2]{\scalebox{#1}{\mbox{\ensuremath{\displaystyle #2}}}}
\newcommand*{\MyScaleBig}{0.95}
\newcommand*{\MyScaleMedium}{0.87}
\newcommand*{\MyScaleSmall}{0.8}
\newcommand*{\MyScaleTiny}{0.7}
\newcommand*{\MyScalePic}{0.6}
\newcommand*{\MyScalePicSmall}{0.4}
\begin{document}
\title{Nikulin involutions and the CHL string}
\author{Adrian Clingher and Andreas Malmendier}
\address{Dept.~\!of Mathematics and Computer Science, University of Missouri -- St.~\!Louis, St.~\!Louis, MO 63121}
\email{clinghera@umsl.edu}
\address{Dept.~\!of Mathematics and Statistics, Utah State University, Logan, UT 84322}
\email{andreas.malmendier@usu.edu}
\begin{abstract}
We study certain even-eight curve configurations on a specific class of Jacobian elliptic K3 surfaces with lattice polarizations of rank ten.  These configurations are associated with K3 double covers, some of which are elliptic but not Jacobian elliptic. Several non-generic cases corresponding to K3 surfaces of higher Picard rank are also discussed. Finally, the results and the construction in question are interpreted in the context of the string dualities linked with the eight-dimensional CHL string.
\end{abstract}
\maketitle
\section{Introduction}
Let $\mathcal{X}$ be an algebraic K3 surface defined over the field of complex numbers. A {\it Nikulin involution} \cite{MR728142,MR544937} is an involution $ \imath_{\mathcal{X}} \colon \mathcal{X} \rightarrow \mathcal{X}$ that satisfies $\imath_{\mathcal{X}}^*(\omega) = \omega $ for any holomorphic two-form $\omega$ on $\mathcal{X}$. The fixed locus of $\imath_{\mathcal{X}}$ always consists of eight distinct points. One takes the quotient of $\mathcal{X}$ by the involution $\imath_{\mathcal{X}}$ and then resolves the eight resulting $A_1$-singularities on the quotient. This procedure, sometime referred in the literature as the {\it Nikulin construction}, leads to a new smooth K3 surface $\mathcal{Y}$, related to $\mathcal{X}$ via a rational double-cover map $\hat{\Phi} \colon \mathcal{X} \dashrightarrow \mathcal{Y}$. Moreover, the exceptional curves $C_1, C_2, \dots, C_8$ resulting on $\mathcal{Y}$ from the singularity resolution form an {\it even-eight configuration}  \cite{MR1922094}, i.e.
$$ \frac{1}{2} \ \left ( C_1 + C_2 + \cdots +  C_8 \right )  \ \in \operatorname{NS}(\mathcal{X}) \ . $$
This configuration of eight curves whose formal sum is in $2\operatorname{NS}(\mathcal{X})$ is known \cite{MR1922094} to determine in full the double cover $\hat{\Phi} \colon \mathcal{X} \dashrightarrow \mathcal{Y}$.
\par {\it Van Geemen-Sarti involutions} \cite{MR2274533,MR2824841} provide a particularly interesting case of Nikulin constructions. In this situation, the K3 surface $\mathcal{X}$ is endowed with a Jacobian elliptic  fibration $\pi_{\mathcal{X}} \colon \mathcal{X} \rightarrow \mathbb{P}^1 $ which, in addition to the trivial section $\sigma$, carries an additional section $\tau$ that makes an element of order two in the Mordell-Weil group $\operatorname{MW}(\pi_{\mathcal{X}},\sigma)$. Fiberwise translations by the order-two section $\tau$ are then known to define an involution on $\imath_{\mathcal{X}}$ on $\mathcal{X}$ -- the Van Geemen-Sarti involution. These involutions are special Nikulin involutions and the Nikulin construction from above carries out, with an extra twist -- the new K3 surface $\mathcal{Y}$ inherits a dual Jacobian elliptic fibration  $\pi_{\mathcal{Y}} \colon \mathcal{Y} \rightarrow \mathbb{P}^1$, as well as its own dual Van Geemen-Sarti involution $\imath_{\mathcal{Y}}$. One obtains a pair of dual geometric two-isogenies between the two K3 surfaces $\mathcal{X}$ and $\mathcal{Y}$. 
\begin{equation}
\label{isog_intro}
 \xymatrix 
{ \mathcal{Y} \ar @(dl,ul) _{\imath_{\mathcal{Y}}}
\ar @/_0.5pc/ @{-->} _{\hat{\Phi}} [rr] &
& \mathcal{X} \ar @(dr,ur) ^{\imath_{\mathcal{X}}}
\ar @/_0.5pc/ @{-->} _{\Phi} [ll] \\
} 
\end{equation}
The even-eight configuration associated to the double cover  $\hat{\Phi} \colon \mathcal{X} \dashrightarrow \mathcal{Y}$ consists of eight rational curves that are each part of a singular fiber of the fibration $\pi_{\mathcal{Y}}$, and are each disjoint from the trivial section of $\pi_{\mathcal{Y}}$. A similar situation holds for the even-eight configuration associated with the double cover $\Phi \colon \mathcal{Y} \dashrightarrow \mathcal{X}$. 
\par One notes that, in the above Van Geemen-Sarti involution context, both surfaces $\mathcal{X}$ and $\mathcal{Y}$ have Picard rank at least ten and, carry what is called a {\it lattice polarization} \cite{MR544937,MR1420220} of type $H \oplus N$. Here, $H$ is the standard rank-two hyperbolic lattice and $N$ is the rank-eight Nikulin lattice~\cite{MR728142}  (generated, in our case, by the eight even-eight curves and half of their formal sum). The converse argument is also true \cite{MR2274533} - the presence of a  $H \oplus N$-polarization on a given K3 surface determines a canonical Van Geemen-Sarti involution. One can actually view the structure of $(\ref{isog_intro})$ as an involution on the ten-dimensional moduli space $\mathcal{M}_{H \oplus N}$ of $H \oplus N$-polarized K3 surfaces. 
\par The generic $H \oplus N$-polarized K3 surface $\mathcal{X}$ carries a Jacobian elliptic fibration $\pi_\mathcal{X} \colon \mathcal{X} \rightarrow \mathbb{P}^1$ with eight singular fibers of Kodaira type $I_2$, eight nodes, and a section of order-two. The {\it canonical} Van Geemen-Sarti even-eight configuration consists then of the eight singular fiber rational curves that do not meet the trivial section (we shall refer to such curves as non-neutral). This even-eight configuration determines, as in $(\ref{isog_intro})$, a dual two-isogeneous Jacobian elliptic $K3$ surface $\mathcal{Y}$ that covers $\mathcal{X}$, via a double-cover $\Phi \colon \mathcal{Y} \dashrightarrow \mathcal{X}$. However, there are other {\it non-canonical} even-eight configurations on $\mathcal{X}$, still  consisting of curves belonging to the singular fibers of $\pi_\mathcal{X}$. These configurations determine additional K3 double covers $\Psi: \mathcal{Z} \dashrightarrow \mathcal{X}$, with the covering K3 surfaces $\mathcal{Z}$ still elliptic but no longer Jacobian elliptic. 
\par It is remarkable that the aforementioned construction provides a mathematical framework for a class of string dualities linked to the so-called \emph{CHL string} in eight dimensions. The CHL string is obtained from the more familiar $E_8 \times E_8$ heterotic string on a torus $T^2$ as a $\mathbb{Z}/2\mathbb{Z}$ quotient. 
F-theory, that is, compactifications of the type-IIB string theory in which the complex coupling varies over a base, is a powerful tool for analyzing the non-perturbative aspects of heterotic string compactifications \cite{MR1409284,MR1412112}.  The simplest F-theory constructions correspond to K3 surfaces that are elliptically fibered over $\mathbb{P}^1$. However, for the CHL string the standard F-theory/heterotic string correspondence has to be modified as follows: Witten found that F-theory compactifications dual to the CHL string  in eight dimensions must carry non-zero flux of an antisymmetric two-form field  \cite{MR1615617}. The presence of this flux freezes eight of the moduli in the physical moduli space, leaving a ten dimensional moduli space. In fact, it turns out that the moduli space of F-theory compactifications dual to the CHL string is a finite covering space of $\mathcal{M}_{H \oplus N}$. Concretely, in order to recover the dual CHL compactification from the Weierstrass model one needs to choose four of the eight points on the base where the fibers of type $I_2$ are located. The double cover of the projective line branched at the four marked points defines the genus-one curve on which the CHL string is compactified. 
\par The identification of points in the physical moduli space corresponding to a specific non-abelian gauge group of the CHL string simplifies considerably, if the elliptically fibered K3 surface admits a second elliptic fibration with two fibers of Kodaira type $I_0^*$ which is no longer Jacobian elliptic. In fact, Witten argued that the moduli space of F-theory compactifications dual to the CHL string is naturally isomorphic to the moduli space of K3 surfaces obtained as genus-one fibrations -- as opposed to an elliptic fibration with section -- with two fibers of type $I_0^*$ and a bisection \cite{MR1615617}. In this paper, we restrict to the case where this second fibration arises in a simple geometric manner, by interchanging the roles of base and fiber in the first fibration. The genus-one fibration then has two singular fibers of Kodaira-type $I_0^*$, four fibers of type $I_2$ and four nodes. It is precisely of the type discussed above, that is, obtained from an elliptic K3 surface with section and order-two section and singular fibers $2I_0^*+4I_2+4I_1$ via a double cover branched along a non-canonical even-eight configurations. On the other hand, elliptic fibrations with section and order-two section and singular fibers $2I_0^*+4I_2+4I_1$ form a six-dimensional natural subspace in a second copy of the moduli space $\mathcal{M}_{H \oplus N}$, namely the subspace where four fibers of type $I_2$ and four fibers of type $I_1$ have coalesced to form two fibers of type $I_0^*$.  This provides a mathematical framework for the F-theory/heterotic string duality linked with the CHL string.
\par This paper is organized as follows. Sections 2 and 3 discuss some general features of the non-canonical even-eight configurations mentioned above, as well as some properties of their associated K3 double covers. As part of this, we develop an explicit computational framework, including a Weierstrass-like normal form for certain non-Jacobian elliptic fibrations. Sections 4 and 5 analyze several non-generic situations of interest which correspond to certain K3 surfaces $\mathcal{Y}$ of high Picard rank.  Finally, the application of these ideas to CHL string and the F-theory/heterotic string correspondence is addressed in Section 6.  In particular, here we determine explicitly the duality map between the moduli of the CHL string in eight dimensions and the dual F-theory models after restricting to a natural six-dimensional subspace of the full moduli space.
%
%
%
\section{Two isogenies and torsors}
\label{setup}
If $\mathcal{E}$ is a smooth elliptic curve over a field $k$ of characteristic zero with a two-torsion point,  then an affine Weierstrass model is given by
\begin{equation}
\label{curve}
 y^2 = x \, (x^2+ b x + a c) 
 \end{equation}
with $a,b,c\in k$. We require the last coefficient $ac$ to be factored over $k$ in order to construct a closely related genus-one curve below. A two-torsion point on $\mathcal{E}$ is given by $\tau: (x,y)=(0,0)$, and we denote the point at infinity by $\sigma$. The discriminant of $\mathcal{E}$ is $\Delta_{\mathcal{E}}=a^2 c^2 (b^2-4ac) \not =0$ and non-vanishing because  $\mathcal{E}$ is assumed smooth. The points $\{ \sigma, \tau\}$ form a subgroup of $\mathcal{E}$ isomorphic to $\mathbb{Z}/2\mathbb{Z}$. The curve $\mathcal{E}$ admits two commuting involutions: the hyperelliptic involution $(x,y) \mapsto (x,-y)$ and the translation by two-torsion. For any point $p \in \mathcal{E}$ the translation by two-torsion is $p \mapsto p + \tau$ where $+$ indicates addition with respect to the group law of the elliptic curve. A computation shows that the involution is given by
\begin{equation}
\label{involution_fiber}
 \imath_{\mathcal{E}}: (x,y) \mapsto (x,y) + (0,0) = \left( \frac{ac}{x}, - \frac{acy}{x^2} \right)
\end{equation} 
for $p \not \in  \{\sigma, \tau\}$  and interchanges $\sigma$ and $\tau$. The holomorphic one-form $dx/y$ on $\mathcal{E}$ is invariant under pullback of $\imath_{\mathcal{E}}$. We obtain a smooth two-isogeneous elliptic curve as quotient $\hat{\mathcal{E}} = \mathcal{E}/\{ \sigma, \tau\}$ which has the Weierstrass model
\begin{equation}
\label{fiber_isog}
 Y^2 = X \big(X^2- 2 b X +  b^2-4ac\big)
\end{equation}
equipped with a two-torsion point $T: (X,Y)=(0,0)$ and the point at infinity $\Sigma$. The discriminant of $\hat{\mathcal{E}}$ is given by $\Delta_{\hat{\mathcal{E}}} =16ac(b^2-4ac)^2 \not =0$.  The degree-two isogeny $\hat{\varphi}:\mathcal{E}\to \hat{\mathcal{E}}$ is given by
\begin{equation}
\label{isog}
 \hat{\varphi}: (x,y) \mapsto (X,Y) = \left( \frac{y^2}{x^2}, \frac{(x^2-ac)y}{x^2} \right) 
\end{equation}
for $(x,y)\not = (0,0)$ and $\hat{\varphi}(\{\sigma, \tau\})=\Sigma$ such that $\hat{\varphi} \circ  \imath_{\mathcal{E}} = \hat{\varphi}$. On $\hat{\mathcal{E}}$ the dual involution $\imath_{\hat{\mathcal{E}}}: P \mapsto P + T$ is given by
\begin{equation}
\label{dual_isog_involution}
 \imath_{\hat{\mathcal{E}}}: (X,Y) \mapsto (X,Y) + (0,0) = \left( \frac{b^2-4ac}{X}, - \frac{(b^2-4ac)Y}{X^2} \right),
\end{equation} 
which in turn leaves the holomorphic one-form $dX/Y$ on $\hat{\mathcal{E}}$ is invariant. The corresponding dual isogeny $\varphi:\hat{\mathcal{E}}\to \mathcal{E}$ is given by
\begin{equation}
\label{dual_isog}
 \varphi: (X,Y) \mapsto (x,y) = \left( \frac{Y^2}{4X^2}, \frac{Y(X^2-b^2+4ac)}{8X^2} \right) 
\end{equation}
for $(X,Y)\not = (0,0)$ and $\varphi(\{\Sigma, T\})=\sigma$. The composition of isogenies $ \varphi\circ\hat{\varphi}$ factors multiplication by two on $\mathcal{E}$.
\par We define two plane algebraic curves $\mathcal{C}$ and $\hat{\mathcal{C}}$ of genus one using the affine coordinates $(u,v)$ and $(U,V)$ and equations
\begin{equation}
 \mathcal{C}: v^2 = u^4 + b  u^2 + a c, \qquad \hat{\mathcal{C}}: V^2 = a  U^4 + b  U^2 + c.
 \end{equation}
 Notice that for the construction of the curve  $\hat{\mathcal{C}}$ a choice of factorization for the coefficient $ac$ in Equation~(\ref{curve}) was required. The discriminants of both curves equal the discriminant of $\hat{\mathcal{E}}$, i.e., $\Delta_{\mathcal{C}}= \Delta_{\hat{\mathcal{C}}} = \Delta_{\hat{\mathcal{E}}} \not =0$. The curves $\mathcal{C}$ and $\hat{\mathcal{C}}$ are smooth and have genus one. They admit the involutions $\imath_{\mathcal{C}}: (u,v) \mapsto (-u, -v)$ and $\imath_{\hat{\mathcal{C}}}: (U,V) \mapsto (-U, -V)$. We have the following:
\begin{lemma}
\label{lemma_reduction}
The genus-one curve $\mathcal{C}$ is isomorphic to $\hat{\mathcal{E}}$ over $k$. The genus-one curve $\hat{\mathcal{C}}$ is isomorphic to $\hat{\mathcal{E}}$ over $k$ if and only if there either exists an element $\alpha \in k$ with $a=\alpha^2$, or there exist elements $e,f,g \in k$ with $b=-e-2af^2$ and $c=af^4+ef^2+g^2$.
\end{lemma}
\begin{proof}
The map $u=Y/(2X)$ and $v=(Y^2+2bX^2-2X^3)/(4X^2)$ with inverse $X=2u^2+b-2v,$ $Y=2u(2u^2-2v+b)$ is an isomorphism between $\mathcal{C}$ and $\hat{\mathcal{E}}$.
\par If $a=\alpha^2$ then $U=u/\alpha$ and $V=v/\alpha$ is a map between $\mathcal{C}$ and $\hat{\mathcal{C}}$ over $k$. Solutions of $aU^4+bU^2+c=g^2$ satisfy $U^2=(-b\pm\sqrt{4ag^2-4ac+b^2})/(2a)$.  Therefore, for a solution to be $k$-rational there must be $e\in k$ such that $4ag^2-4ac+b^2=e^2$. The curve $\hat{\mathcal{C}}$ has the $k$-rational point $(U,V)=(f,g)$ if and only $2af^2=-b \pm e$. We then have $b=-e-2af^2$ and $g^2+f^2e+af^4=c$. That is, the curve $\hat{\mathcal{C}}$ has the $k$-rational point $(U,V)=(f,g)$ if there exist $e,f,g \in k$ with $b=-e-2af^2$ and $c=af^4+ef^2+g^2$. If there is a $k$-rational point on $\hat{\mathcal{C}}$, $\mathcal{E}$ and $\hat{\mathcal{C}}$ are isomorphic.
\end{proof}
\begin{lemma}\label{lemma_equivariance}
The map $\hat{\mathcal{E}} \dasharrow \mathcal{C}$ given by $u=Y/(2X)$ and $v=(b^2-4ac-X^2)/(4X)$ is equivariant with respect to the hyperelliptic involution $(X,Y) \mapsto (X,-Y)$ on $\hat{\mathcal{E}}$ and the involution $(u,v) \mapsto (-u,v)$ on $\mathcal{C}$. Similarly, the map is equivariant with respect to the involution $ \imath_{\hat{\mathcal{E}}}$ on $\hat{\mathcal{E}}$ and the involution $(u,v) \mapsto (-u,-v)$ on $\mathcal{C}$.
\end{lemma}
\begin{proof}
The lemma follows by explicit computation.
\end{proof}
We define a two-to-one map $\psi: \hat{\mathcal{C}} \to \mathcal{E}$ by
\begin{equation} \label{eq1}
\begin{split}
 \psi: (U,V)  & \mapsto (x,y) = \big(a\, U^2, a \, UV \big),
 \end{split}
\end{equation}
such that $\psi\circ \imath_{\mathcal{C}} = \psi$.  We have the following:
\begin{proposition}
\label{Jac_prop}
The Jacobian variety of the curve $\hat{\mathcal{C}}$ is isomorphic to $\hat{\mathcal{E}}$, i.e., $\operatorname{\operatorname{Jac}}(\hat{\mathcal{C}}) \cong \hat{\mathcal{E}}$.
\end{proposition}
\begin{proof}
The smooth varieties $\hat{\mathcal{C}}$, $\hat{\mathcal{E}}$, and $\mathcal{C}$ have genus one and  a basis of their holomorphic one-forms given by $\frac{dU}{V}$, $\frac{dX}{Y}$, and $\frac{dx}{y}$, respectively. A computation shows that $\varphi^*\big(\frac{dx}{y}\big) = 2 \frac{dX}{Y}$ and $\psi^*\big(\frac{dx}{y}\big) = 2 \frac{dU}{V}$. It follows that the Jacobians of $\hat{\mathcal{E}}$ and $\hat{\mathcal{C}}$ are isomorphic.  But $\hat{\mathcal{E}}$ is an elliptic curve and isomorphic to its Jacobian.
\end{proof}
\section{Application to elliptic fibrations}\label{construction}
$\mathcal{X}$ is an elliptic surface over $\mathbb{P}^1$ if $\mathcal{X}$ is a proper smooth minimal surface  together with a proper morphism $\pi_{\mathcal{X}}: \mathcal{X} \rightarrow \mathbb{P}^1$ with connected fibers such that almost all fibers are smooth curves of genus one.  It will always be assumed that $\mathcal{X}$ does not have exceptional curves of the first kind in the fibers or multiple singular fibers. The $j$-function is the rational function on $\mathbb{P}^1$ whose values in $\mathbb{C} \cup \{\infty\}$ are the $J$-invariants of the fibers $\mathcal{X}_{[t_0:t_1]}$ over $[t_0:t_1] \in \mathbb{P}^1$. It follows that a smooth minimal elliptic surfaces over $\mathbb{P}^1$ are in one-to one correspondence with genus-one curves over $\mathbb{C}(t_0,t_1)$ or \emph{genus-one fibrations}. An elliptic surface  $\pi_{\mathcal{X}}: \mathcal{X} \rightarrow \mathbb{P}^1$ is said to have a section if there exists a global section $\sigma: \mathbb{P}^1 \rightarrow \mathcal{X}$ such that $\pi_{\mathcal{X}}\circ \sigma=\operatorname{id}$. The section corresponds to a $\mathbb{C}(t_0,t_1)$-rational point on the generic fiber which in turn becomes an elliptic curve. 
\par We will restrict ourselves to elliptic K3 surfaces with section, i.e., triples $(\mathcal{X},\pi_{\mathcal{X}},\sigma)$, whose minimal Weierstrass equation is given by
\begin{equation}
  y^2 = x^3 + f x + g ,
\end{equation} 
where $f$ and $g$ are homogeneous polynomials in the variables $[t_0:t_1] \in \mathbb{P}^1$ of degrees $8$ and $12$, respectively. If necessary, we will also use the affine coordinate $t=t_0/t_1$ for $t_1\not =0$. We denote the Mordell-Weil group of sections by $\operatorname{MW}(\pi_{\mathcal{X}}, \sigma)$ or $\operatorname{MW}(\pi_{\mathcal{X}})$ for short. We denote N\'eron-Severi  lattice by $\operatorname{NS}(\mathcal{X})=H^2(\mathcal{X},\mathbb{Z}) \cap H^{1,1}(\mathcal{X})$ and the transcendental lattice by $T_{\mathcal{X}} = \operatorname{NS}(\mathcal{X})^{\perp} \in H^2(\mathcal{X},\mathbb{Z})$.
\par In the following, we will restrict to a situation where $\mathbb{Z}/2\mathbb{Z} \subset \operatorname{MW}(\pi_{\mathcal{X}})$, that is, we will assume that there is a section $\tau: \mathbb{P}^1\rightarrow \mathcal{X}$ different from $\sigma$ such that $2  \tau=\sigma$. A minimal Weierstrass model for a K3 surface with fiber $\mathcal{X}_{[t_0:t_1]}\cong \mathcal{E}$ is given by
\begin{equation}
\label{Xsurface}
 \mathcal{X}_{[t_0:t_1]}: \quad y^2 = x \, \left(x^2 + b   x +  a c\right) ,
\end{equation} 
where $a, b, c$ are homogeneous polynomials in $k(t_0,t_1)$ such that $b$ and $ac$ have degrees $4$ and $8$, respectively. As in the case of elliptic curves, we require the last coefficient $ac$ to be factored over $\mathbb{C}(t_0,t_1)$. The singular fibers of $\mathcal{X}$ are located over the support of $\Delta_{\mathcal{X}}=a^2c^2 (b^2-4ac)$. The two-torsion section $\tau$ is given by $\tau: [t_0:t_1] \mapsto (x,y)=(0,0)$. The holomorphic two-form on $\mathcal{X}$  is $\omega_{\mathcal{X}}=dt\wedge dx/y$.  Because of the two-torsion, the fibration $\pi_{\mathcal{X}}$ cannot have singular fibers of type $II, II^*$ or $IV, IV^*$. Without loss of generality, we will assume that $\deg{a} \le \deg{c}$, $a$ is $\emph{not constant}$, and the \emph{degrees are even}, i.e., $\deg{(a)}=2k$ and $\deg{(c)}=8-2k$ for $k \in \{1, 2\}$.
\subsection{Van Geemen-Sarti involutions}
A \emph{Nikulin involution} on a K3 surface $\mathcal{X}$ is a symplectic involution $\imath_{\mathcal{X}}: \mathcal{X} \to \mathcal{X}$, i.e., an involution such that $\imath_{\mathcal{X} }^*(\omega)=\omega$. If a Nikulin involution exists on a K3 surface $\mathcal{X}$, then it necessarily has eight fixed points, and the minimal resolution of the quotient surface with eight rational double points is another K3 surface $\mathcal{Y}=\widehat{\mathcal{X}/\{1,\imath_{\mathcal{X} }\}}$ \cite{MR544937}. Special Nikulin involution are obtained in our situation: the fiberwise translation by the two-torsion section $p \mapsto p+\tau$ for all  $p \in \mathcal{X}_{[t_0:t_1]}$ extends to a Nikulin involution $\imath_{\mathcal{X} }$ on $\mathcal{X}$, called \emph{van Geemen-Sarti} involution, and  the isogeny in Equation~(\ref{isog}) extends  to a degree-two rational map $\hat{\Phi}: \mathcal{X} \dasharrow \mathcal{Y}$. To prove this, it suffices to check that the fiberwise involution in  Equation~(\ref{involution_fiber}) preserves the  holomorphic two-form. It turns out that the new K3 surface $\mathcal{Y}$ admits again an elliptical fibration $\pi_{\mathcal{Y}}:\mathcal{Y}\to\mathbb{P}^1$ with section $\Sigma$. In fact, when applying the two-isogeny fiberwise, we obtain a Weierstrass model for a K3 surface from Equation~(\ref{fiber_isog}) with typical fiber $\mathcal{Y}_{[t_0:t_1]}\cong \hat{\mathcal{E}}$ given by
\begin{equation}
\label{Ysurface}
 \mathcal{Y}_{[t_0:t_1]}: \quad Y^2 = X \big(X^2- 2 b   X +  b^2-4ac\big).
\end{equation}
The singular fibers of $\mathcal{Y}$ are located over the support of $\Delta_{\mathcal{Y}}=16ac\,(b^2-4ac)^2$. The holomorphic two-form on $\mathcal{Y}$ is $\omega_{\mathcal{Y}}=dt\wedge dX/Y$.  We observe that the K3 surface $\mathcal{Y}$ satisfies $\mathbb{Z}/2\mathbb{Z} \subset \operatorname{MW}(\pi_{\mathcal{Y}})$ with a two-torsion section $T$ given by $T: [t_0:t_1] \mapsto (X,Y)=(0,0)$. Therefore, the surface $\mathcal{Y}$ is equipped with a van Geemen-Sarti involution $\imath_{\mathcal{Y}}$ covering the rational map $\Phi$ extending the fiberwise dual isogeny $P \mapsto P + T$ for all $P \in \mathcal{Y}_t$ in Equation~(\ref{dual_isog}) (cf.~\!\cite{MR2824841}).  The situation is summarized in the following diagram: 
\begin{equation}
\label{pic:VGS}
\xymatrix{
\mathcal{X} \ar @(dl,ul) _{\imath_{\mathcal{X}}} \ar [dr] _{\pi_{\mathcal{X}}} \ar @/_0.5pc/ @{-->} _{\hat{\Phi}} [rr]
&
& \mathcal{Y} \ar @(dr,ur) ^{\imath_{\mathcal{Y}}} \ar [dl] ^{\pi_{\mathcal{Y}}} \ar @/_0.5pc/ @{-->} _{\Phi} [ll] \\
& \mathbb{P}^1 }
\end{equation}
\subsubsection{The generic case}
\label{sec:generic}
For generic polynomials $a, b, c$, the fibrations $\pi_{\mathcal{X}}$ and $\pi_{\mathcal{Y}}$ form two ten-dimensional lattice polarized families of K3 surfaces. The surface $\mathcal{X}$ with $\Delta_{\mathcal{X}}=a^2 c^2 (b^2-4ac)$ has eight fibers of Kodaira type  $I_2$ at the roots of $ac$ and eight fibers of type $I_1$ at the roots of $b^2-4ac$.  The non-neutral components of the reducible fibers on $\mathcal{X}$, i.e., components which do not meet the zero section $\Sigma$, form a sub-lattice $\langle -2 \rangle^{\oplus 8}$ of rank eight perpendicular to  the lattice generated by the section and smooth fiber.  Similar statements hold for $\mathcal{Y}$. It was proven in~\cite{MR2274533} that the N\'eron-Severi groups $\operatorname{NS}(\mathcal{X})$ and $\operatorname{NS}(\mathcal{Y})$ are in fact given by
\begin{equation}
 \operatorname{NS}(\mathcal{X}) \cong \operatorname{NS}(\mathcal{Y}) \cong H \oplus N .
\end{equation} 
Here, $N$ is the rank-eight Nikulin lattice, i.e., the lattice spanned by eight curves $N_i$ for $1\le i \le8$ (with $N_i \cdot N_j=0$ for $i \not =j$ and $N_i \cdot N_i=-2$) and the class $\hat{N} =(N_1 + \dots + N_8)/2$.  In particular, the Picard ranks are ten, i.e., $\rho_X=\rho_Y=10$. It was also proven in~\cite{MR2274533} that the transcendental lattices $T_{\mathcal{X}}$ and $T_{\mathcal{Y}}$ satisfy
\begin{equation}
 T_{\mathcal{X}} \cong T_{\mathcal{Y}} \cong H \oplus (H \oplus N) \;.
\end{equation} 
\par We will show that in the reducible fibers of the K3 surface $\mathcal{X}$ there are eight disjoint smooth rational curves $\lbrace C_1, \dots, C_8\rbrace$ consisting only of exceptional divisors,  that form an even eight, i.e., satisfy $C_1 + \dots + C_8  \in 2 \, \operatorname{NS}(\mathcal{X})$. One then obtains the surface $\mathcal{Y}$ as double cover of $\mathcal{X}$ branched along the even eight after the inverse images of the curves $C_i$ for $1\le i \le 8$ are blown down \cite{MR2306633}. We have the following:
\begin{proposition}
\label{prop_EE}
For generic polynomials $a, b, c$, the maps $\hat{\Phi}$ and $\Phi$ are double covers branched along two even eights on $\mathcal{Y}$ and $\mathcal{X}$, respectively. The even eights are the eight disjoint non-neutral components of the reducible fibers.
\end{proposition}
\begin{proof}
It suffices to give a proof for $\Phi: \mathcal{Y} \dasharrow \mathcal{X}$. The surface $\mathcal{X}$ with $\Delta_{\mathcal{X}}=a^2 c^2 (b^2-4ac)$ has eight fibers of type $I_2$ at the roots of $ac$ and eight fibers of type $I_1$ at the roots of $b^2-4ac$. A double cover branched along eight disjoint rational curves that are all non-neutral components of reducible fibers is a fiberwise isogeny covered by a van Geemen-Sarti involution~\cite{MR2824841}. In particular, it is a Nikulin involution and the eight disjoint rational curves are an even eight. It only remains to show that the branched double cover is the map $\Phi$. The inverse image of a  node on $\mathcal{X}$ that does not intersect the branching locus is a reducible fiber of type $A_1$, and the inverse image of a reducible fiber of type $A_1$ with its non-neutral component in the branch locus becomes a node, i.e., a fiber of type $I_1$. The branch locus does not meet the section $\sigma$. The preimage of the section $\sigma$ are two sections $\Sigma$ and $T$. On $\mathcal{X}$ the singularities in the fibers of type $I_2$ are located over the roots of $ac$ at the points $(x,y)=(0,0)$. Therefore, the two-torsion section $\tau$ meets all non-neutral components of the reducible $A_1$ fibers which are in the branch locus. The singularities in the fibers of type $I_1$ are located over the roots of $b^2-4ac$ at the points $(x,y)=(-b/2,0)$ which is not met by the two-torsion section. Therefore, the preimage of the section $\tau$ is not a section. As fiberwise two-isogeny changes the singular fibers according to $I_1 \leftrightarrow I_2$ and the new surface has a section and a two-torsion section, it is the double cover branched at the given even eight. 
\end{proof}
\begin{remark}
\label{rem:Hirzebruch}
A projective model for the surface $\pi_{\mathcal{X}}: \mathcal{X} \to \mathbb{P}^1$ is given by
\begin{equation}
\label{K3_X}
 \mathcal{X}_{[t_0:t_1]}: \quad y^2 z  = x \Big( x^2 + b \, x z+ ac  \, z^2\Big ), \quad  \Delta_{\mathcal{X}}=a^2 c^2 (b^2-4ac).
\end{equation} 
with $[x:y:z] \in \mathbb{P}^2$. There are two $\mathbb{C}^*$ actions on Equation~(\ref{K3_X}), namely for $\lambda \in \mathbb{C}^*$,
\begin{equation}
\begin{split}
 \mu_1: \Big(t_0,t_1,x,y,z\Big) & \mapsto \Big( t_0,t_1, \lambda x, \lambda y, \lambda z\Big),\\
 \mu_2: \Big(t_0,t_1,x,y,z\Big) & \mapsto \Big( \lambda t_0, \lambda t_1, x, y,  \lambda^{-4} z\Big).
\end{split}
\end{equation}
By means of the projection $\big(t_0,t_1,x,y,z\big) \mapsto \big(t_0,t_1,x,z\big)$, the surface $\mathcal{X}$ is a double cover of the Hirzebruch surface $\mathbb{F}_4$. In fact, the quotient of $\mathcal{X}$ by the hyperelliptic involution is a rationally ruled surface whose minimal resolution is $\mathbb{F}_4$ with eight blow-ups. The same applies to the surface $\mathcal{Y}$.
\end{remark}
\subsection{Elliptic fibrations without sections}
For the elliptically fibered K3 surfaces $\pi_{\mathcal{X}} : \mathcal{X} \to \mathbb{P}^1$ and $\pi_{\mathcal{Y}} : \mathcal{Y} \to \mathbb{P}^1$ with section $\sigma$ and $\Sigma$, respectively, related by van Geemen-Sarti involutions described above, we define a genus-one fibration $\pi_{\mathcal{Z}}: \mathcal{Z} \to \mathbb{P}^1$ by specifying for each $[t_0:t_1] \in \mathbb{P}^1$ a fiber $\mathcal{Z}_{[t_0:t_1]}\cong \hat{\mathcal{C}}$ given by the equation
\begin{equation}
\label{Zsurface}
 \mathcal{Z}_{[t_0:t_1]}: \quad V^2 = a  U^4 + b U^2 + c ,
\end{equation}
where the polynomials $a, b, c$ are the same as in the definition of $\mathcal{X}$ and $\mathcal{Y}$, i.e., $a, b, c$ are homogeneous polynomials in $\mathbb{C}(t_0,t_1)$ of degree $2k$, $4$, and $8-2k$, respectively, with $k \in \{1,2\}$. The singular fibers of $\mathcal{Z}$ are located over the support of $\Delta_{\mathcal{Z}}=16ac(b^2-4ac)^2$ whence the base loci of the singular fibers for $\mathcal{X}, \mathcal{Y}, \mathcal{Z}$ are the same.  An involution $\imath_{\mathcal{Z}}$ on $\mathcal{Z}$ is given by $\imath_{\mathcal{Z}} : ([t_0:t_1],(U,V)) \mapsto ([t_0:t_1],(-U,-V))$. We have the following:
\begin{proposition}
\label{prop_K3}
The genus-one fibration $\pi_{\mathcal{Z}}: \mathcal{Z} \to \mathbb{P}^1$ has the following properties:
\begin{enumerate}
 \item $\mathcal{Z}$ is a K3 surface with holomorphic two-form $\omega_{\mathcal{Z}} = dt \wedge dU/V$.
 \item $\mathcal{Z}$ is a double cover of the Hirzebruch surface $\mathbb{F}_{2-k}$.
 \item  For each $[t_0:t_1] \in \mathbb{P}^1 - \operatorname{supp}(\Delta_{\mathcal{Z}})$ it follows $\operatorname{Jac}(\mathcal{Z}_{[t_0:t_1]})\cong \mathcal{Y}_{[t_0:t_1]}$. 
 \item The fibrations $\pi_{\mathcal{Y}}, \pi_{\mathcal{Z}}$ have the same singular fibers.
 \item The map in Equation~(\ref{isog}) extends to degree-two rational map 
 $\Psi:  \mathcal{Z} \dasharrow \mathcal{X}$ that preserves the base points of every fiber and satisfies $\Psi \circ \imath_{\mathcal{Z}} =\Psi$.
\item $\imath_{\mathcal{Z}}$ is a Nikulin involution such that $\mathcal{X}=\widehat{\mathcal{Z}/\{1,\imath_{\mathcal{Z}}\}}$.
\item For $k=2$, $\mathcal{Z}$ admits a second genus-one $\check{\pi}_{\mathcal{Z}} : \mathcal{Z} \to \mathbb{P}^1$ with singular fibers $8 I_2 + 8 I_1$ coming from interchanging the roles of base and fiber coordinates.

\end{enumerate}
\end{proposition}
\begin{proof}
(1) The polynomials $a, b$, and $c$ are homogeneous polynomials in the variables $[t_0:t_1] \in \mathbb{P}^1$ of degree $2k$, $4$, and $8-2k$, respectively. By setting $V \mapsto V/W^2$, $U \mapsto U/W$, we obtain the homogeneous equation
\begin{equation}
\label{K3}
 \mathcal{Z}_{[t_0:t_1]}: \quad V^2 = a U^4 + b  U^2 W^2 + c  W^4 .
\end{equation}
There are two $\mathbb{C}^*$-actions on Equation~(\ref{K3}), namely for $\lambda \in \mathbb{C}^*$,
\begin{equation}
\label{scaling}
\begin{split}
 \mu_1: \Big(t_0,t_1,U,V,W\Big) & \mapsto \Big( t_0,t_1, \lambda U, \lambda^2 V, \lambda W\Big),\\
 \mu_2: \Big(t_0,t_1,U,V,W\Big) & \mapsto \Big( \lambda t_0, \lambda t_1, U, \lambda^{k} V,  \lambda^{k-2} W\Big).
\end{split}
\end{equation}
With respect to the action $\mu_1$ (or $\mu_2$), Equation~(\ref{K3}) has degree $4$ (or $2k$) which equals the sum of the weights for the variables $(t_0,t_1,U,V,W)$. Moreover, the action $\mu_1$ leaves the homogeneous coordinates on $\mathbb{P}^1$ invariant and $\mathcal{Z}$ is a genus-one fibration over $\mathbb{P}^1$. The coordinates $[U:V:W] \in \mathbb{P}(1,2,1)$ are sections of  line bundles over the base $\mathbb{P}^1$ specified by means of their transformation properties under the group $\mu_2$.  Since the condition is satisfied that the total weight of Equation~(\ref{K3}) equals  the sum of weights of the defining variables for each acting $\mathbb{C}^*$, a surface with trivial canonical class is obtained by removing the loci  $\lbrace t_0 = t_1 = 0\rbrace$, $\lbrace U = V = W = 0\rbrace$ from the solution set of Equation~(\ref{K3}), and taking the quotient by $(\mathbb{C}^*)^2$. Therefore, $\mathcal{Z}$ is a K3 surface, and the holomorphic two-form invariant under the two $\mathbb{C}^*$ actions is given by
\begin{equation}
 \Big(t_1 dt_0-t_0 dt_1\Big) \wedge \frac{W \, dU - U \, dW}{V} .
\end{equation}
\par (2) By means of the projection $\big(t_0,t_1,U,V,W\big) \mapsto \big(t_0,t_1,U,W\big)$, the surface $\mathcal{Z}$ is a double cover of the Hirzebruch surface $\mathbb{F}_{2-k}$.
\par (3) This follows from Proposition~\ref{Jac_prop}.
\par (4) The discriminants of $\mathcal{Y}$ and $\mathcal{Z}$ are identical.  In particular, the fiber $\mathcal{Z}_t$ is smooth if and only if $\mathcal{Y}_t$ is.  The Kodaira classification of singular fibers is the same for elliptically fibered  surfaces without section with the following difference: there is the possibility with $I_n$ fibers with $n \ge 0$  of having multiple fibers. Since $\mathcal{Z}$ does not have multiple fibers and each fiber $\mathcal{Z}_t$ is isomorphic to $\mathcal{Y}_t$ the claim follows.
\par (5) This follows from the fact that $\mathcal{Z}$ is a genus-one fibration without multiple fibers and the map in each fiber is given by Equation~(\ref{eq1}).
\par (6) The involution $\imath_{\mathcal{Z}}$ leaves the holomorphic two-form $\omega_{\mathcal{Z}}$ invariant and interchanges the sheets of $\Psi$.
\par (7) For $k=2$ the roles of $[t_0:t_1]$ and $[U:W]$ in the $\mathbb{C}^*$-actions in Equation~(\ref{scaling}) can be interchanged. We obtain a second projection $\check{\pi}_{\mathcal{Z}} : \mathcal{Z} \to \mathbb{P}^1$ by $(t_0,t_1,U,V,W) \mapsto [U:W]$. Because the degree of $a,b,c$ equals four for $k=2$, the projection induces another genus-one fibration.
\end{proof}
We have the following:
\begin{theorem}
\label{thm}
\begin{enumerate} 
\item[]
\item For each pair of elliptically fibered K3 surfaces with sections $\pi_{\mathcal{X}} : \mathcal{X} \to \mathbb{P}^1$ and $\pi_{\mathcal{Y}} : \mathcal{Y} \to \mathbb{P}^1$ related by van Geemen-Sarti involutions given by Equations~(\ref{Xsurface}) and~(\ref{Ysurface}), there exists a K3 surface $\mathcal{Z}$ together with a genus-one fibration $\pi_{\mathcal{Z}} : \mathcal{Z} \to \mathbb{P}^1$ given by Equation~(\ref{Zsurface}) with the same singular fibers as $\pi_{\mathcal{Y}}$.
\item If there exist
\begin{equation}
\label{condition}
\begin{split}
 \exists \, \alpha \in k(t_0,t_1): a=\alpha^2 \;\; \text{or} \quad \exists \, e,f,g \in k(t_0,t_1): \left\lbrace \begin{array}{ll} b&=e +2af^2, \\c&=f^2(e+af^2)+g^2, \end{array}\right.
\end{split}
\end{equation}
 then $\pi_{\mathcal{Z}}$ admits a section. In particular, the surfaces $\mathcal{Y}$ and $\mathcal{Z}$ are isomorphic as elliptically fibered surfaces with section, i.e., $\mathcal{Y} \cong \mathcal{Z}$, and the Picard numbers (which coincide) are bigger than ten, i.e., $\rho_{\mathcal{X}}=\rho_{\mathcal{Y}}=\rho_{\mathcal{Z}}>10$.
\item If condition~(\ref{condition}) is not satisfied, then the genus-one fibration $\pi_{\mathcal{Z}}$ does not admit a section. In particular,  $Y$ and $Z$ are different K3 surfaces.
\end{enumerate}
\end{theorem}
\begin{proof}
(1) follows from Proposition~\ref{prop_K3}. (2) The statement follows from Lemma~\ref{lemma_reduction} except for the statement about the Picard number. The condition in Equation~(\ref{condition}) is satisfied if $a=\alpha^2$ or $c=g^2$. On $\mathcal{X}$ we then have $k$ or $4-k$ fibers of Kodaira type $I_4$ instead of $2k$ or $8-2k$ fibers of Kodaira type $I_2$, and the Picard number is at least $10+k$ or $14-k$ for $k=1, 2$. If the condition in Equation~(\ref{condition}) is satisfied for $f \not =0$, then there is a non-torsion section $(x,y)=(-af^2,iafg)$ on $\mathcal{X}$. In particular, if either of the conditions in Equation~(\ref{condition}) is satisfied, the Picard number is strictly bigger than ten. A theorem by Inose states \cite[Cor.~1.2]{MR0429915} that two K3 surfaces defined over a number field and connected by a rational map of finite degree have the same Picard rank. (3) In the generic case, $\mathcal{Y}$ has singular fibers $8 I_2 + 8 I_1$ and $\operatorname{MW}(\pi_{\mathcal{Y}})=\mathbb{Z}/2\mathbb{Z}$ it follows that  $\det{q_{T_{\mathcal{Y}}}}=\det{q_{\operatorname{NS}(\mathcal{Y})}}=2^8/2^2=2^6$.  The K3 surface $\mathcal{Z}$ does not admit as section, but it is obvious that it admits at least two bi-sections, obtained by setting $U=0, W=1$ or $U=1,W=0$ in Equation~(\ref{K3}). In particular, the multi-section index of $\mathcal{Z}$ equals $l=2$. By construction, the relative Jacobian fibration of the genus-one fibration $\pi_{\mathcal{Z}}$ is the elliptic fibration with section $\pi_{\mathcal{Y}}$. It was proven in \cite{MR986969} that the N\'eron-Severi lattice $\operatorname{NS}(\mathcal{Z})$ is a sub-lattice of the lattice $\operatorname{NS}(\mathcal{Y})$ of index $l$. It follows that the determinants of the discriminant forms are related by $\det{q_{\operatorname{NS}(\mathcal{Z})}}= l^2 \det{q_{\operatorname{NS}(\mathcal{Y})}}$.  In particular, $Y$ and $Z$ cannot be isomorphic.
\end{proof}
\subsubsection{The generic case}
We have the following:
\begin{proposition}
\label{prop_EEP}
For generic polynomials $a, b, c$, the map $\Psi:\mathcal{Z} \dasharrow \mathcal{X}$ is the double cover branched along the even eight on $\mathcal{X}$ given by the neutral components of the reducible fibers over $a=0$ and the non-neutral components over $c=0$. In particular, $\mathcal{Z}$ does not have a section.
\end{proposition}
\begin{proof}
We write down projective models for $\mathcal{X}, \mathcal{Z}, \mathcal{Y}$ with coordinates $[x:y:z], [X:Y:Z]\in \mathbb{P}^2$, $[U:V:W]\in \mathbb{P}(1,2,1)$ for the fibers and $[t_0:t_1]\in \mathbb{P}^1$ for the base from the equations
\begin{equation}
\label{general_model}
\scalemath{\MyScaleBig}{
\begin{array}{lll}
 \mathcal{Y}_{[t_0:t_1]}:& Y^2 Z  = X \Big( X^2 - 2 b  X Z + \big(b^2-4ac\big)  Z^2\Big), &\Delta_{\mathcal{Y}}=\Delta_{\mathcal{Z}},\\[0.5em]
 \mathcal{Z}_{[t_0:t_1]}:& V^2  = a  U^4 + b  U^2W^2 + c W^4, &\Delta_{\mathcal{Z}}=16 \, ac \, (b^2-4ac)^2, \\[0.5em]
 \mathcal{X}_{[t_0:t_1]}:& y^2 z  = x \big( x^2 + b  x z+ ac  z^2\big),   &\Delta_{\mathcal{X}}=a^2 c^2 (b^2-4ac).
\end{array}}
\end{equation} 
On $\mathcal{Y}$ there are sections $\Sigma: [X:Y:Z]=[0:1:0]$ and $T: [X:Y:Z]=[0:0:1]$, and a bisection $\Upsilon: [X:Y:Z]=[b\pm 2\sqrt{ac}:0:1]$. On $\mathcal{Z}$ there are two bi-sections $\kappa: [U:V:W]=[1:\pm\sqrt{a}:0]$ and $\upsilon: [U:V:W]=[0:\pm \sqrt{c}:1]$. On $\mathcal{X}$ there are sections $\sigma: [x:y:z]=[0:1:0]$ and $\tau: [x:y:z]=[0:0:1]$. The rational maps $\Psi: \mathcal{Z} \dasharrow \mathcal{X}$ and $\Phi:\mathcal{Y} \dasharrow \mathcal{Z}$ -- for $(X,Y)\not =(0,0)$ and $(Y,Z)\not =(0,0)$ -- are given by
\begin{equation}
\begin{split}
  [x:y:z] & =\Phi([X:Y:Z])= \left[2 Y^2 Z, Y\Big(X^2- \big(b^2-4ac\big) Z^2\Big): 8X^2Z\right], \\
  [x:y:z] & =\Psi([U:V:W])= [aU^2W: a U V : W^3].
\end{split}
\end{equation} 
The fibers of type $I_2$ on $\mathcal{X}$ are located over $ac=0$ with singular point $[x:y:z]=[0:0:1]\in \tau$. Therefore, the section $\tau$ intersects all non-neutral components of the reducible $A_1$ fibers. By definition of $\Phi$, the preimages of $\sigma$ are $\Sigma$ and $T$; one checks that the preimage of $\tau$ is $\Upsilon$. The preimage of $\sigma$ under $\Psi$ is $\kappa$, the preimage of $\tau$ is $\upsilon$.
\par The fibers of type $I_1$ on $\mathcal{Y}$ are located over $ac=0$ with singular point $[X:Y:Z]=[b:0:1] \in \Upsilon - T$. It follows that $\mathcal{Y}$ is obtained as double cover branched along the components of the reducible fibers on $\mathcal{X}$ intersected by $\tau$. $\mathcal{Z}$ has the same singular fibers as $\mathcal{Y}$. However, over $a=0$ the fiber $\mathcal{Z}_t$ has the form $V^2=W^2(b\, U^2+c \, W^2)$ with singular point $[1:0:0] \in \kappa -\upsilon$, and over $c=0$ the fiber $\mathcal{Z}_t$ has the form $V^2=U^2(a\, U^2+b \, W^2)$ with singular point $[0:0:1] \not \in \upsilon - \kappa$. Therefore, $\mathcal{Z}$ is a double cover branched along the $2k$ neutral components of the reducible fibers over $a=0$ and the $8-2k$ non-neutral components over $c=0$. 
\end{proof}
\noindent
The different even eights chosen in $\mathcal{X}$ as the branch locus to obtain the double covers $\mathcal{Y}$ (green) and $\mathcal{Z}$ (yellow) in Proposition~\ref{prop_EEP} are shown in the following diagram:
\begin{center}
\includegraphics[scale=\MyScalePic]{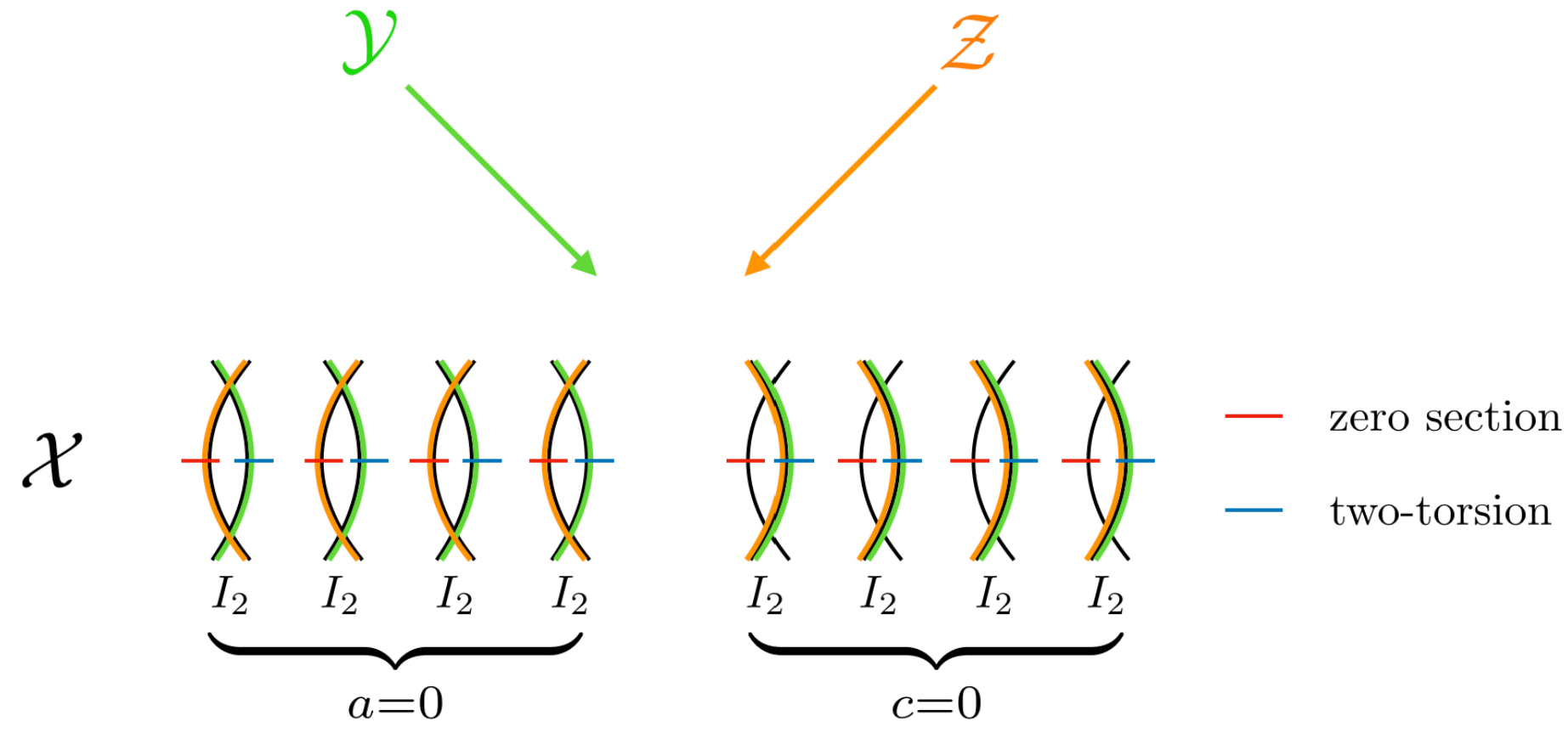}
\end{center}
\section{Examples with Picard numbers $14$ and $18$}
We analyze different situations in which fibers of Kodaira type $I_2$ in the generic fibration with eight $I_2$ and eight $I_1$ coalesce to form four fibers of Kodaira type $I_4$ or $I_0^*$. These configurations occur for $k=2$ in Theorem~\ref{thm} and Proposition~\ref{prop_K3} in Picard rank $14$ and $18$.
\subsection{Four $I_4$ fibers in Picard rank $\rho=14$}
We specialize Theorem~\ref{thm} and assume that $c= a$ and $a, b$ are homogeneous polynomials of degree four in $\mathbb{C}(t_0,t_1)$ with no common factor or repeated roots. We obtain projective models for K3 surfaces $\mathcal{X}, \mathcal{Y}, \mathcal{Z}$ with coordinates $[x:y:z], [X:Y:Z]\in \mathbb{P}^2$, $[U:V:W]\in \mathbb{P}(1,2,1)$ for the fibers and $[t_0:t_1]\in \mathbb{P}^1$ for the base from the equations
\begin{equation}
\scalemath{\MyScaleBig}{
\begin{array}{lll}
 \mathcal{Y}_{[t_0:t_1]}:&  Y^2 Z  = X \Big(X - (b - 2a)Z\Big) \Big(X - (b +2a)Z\Big), &\Delta_{\mathcal{Y}}=\Delta_{\mathcal{Z}},\\[0.5em]
 \mathcal{Z}_{[t_0:t_1]}:& V^2  =  a U^4 + b U^2W^2 + a W^4, & \Delta_{\mathcal{Z}}=2^4 a^2  (b^2-4a^2)^2, \\[0.5em]
 \mathcal{X}_{[t_0:t_1]}:& y^2 z  = x \big( x^2 + b  x z+  a^2 z^2\big), & \Delta_{\mathcal{X}}=a^4 (b^2-4a^2).
\end{array}}
\end{equation} 
Condition~(\ref{condition}) is not satisfied, and the genus-one fibration $\pi_{\mathcal{Z}}$ has no sections. The elliptic fibrations with section $\pi_{\mathcal{X}}$ and $\pi_{\mathcal{Y}}$ have the singular fibers $4 I_4+8I_1$ and $\operatorname{MW}(\pi_{\mathcal{X}})=\mathbb{Z}/2\mathbb{Z}$, and $12 I_2$ and $\operatorname{MW}(\pi_{\mathcal{Y}})=(\mathbb{Z}/2\mathbb{Z})^2$, respectively. Accordingly, we have $\rho_{\mathcal{X}} = \rho_{\mathcal{Y}}= \rho_{\mathcal{Z}} = 14$. We have the following:
\begin{proposition}
\label{prop_EEP_I4}
In the situation described above, the map $\Phi:\mathcal{Y} \dasharrow \mathcal{X}$ is the double cover branched along the even eight on $\mathcal{X}$ given by the components of the reducible fibers of type $A_3$ over $a=0$ not meeting the sections $\sigma$ or $\tau$. Similarly, the map $\Psi: \mathcal{Z} \dasharrow \mathcal{X}$ is the double cover branched along the even eight on $\mathcal{X}$ given by the components meeting sections $\sigma$ or $\tau$.
\end{proposition}
\begin{proof} 
On $\mathcal{X}$ there are sections $\sigma: [x:y:z]=[0:1:0]$ and $\tau: [x:y:z]=[0:0:1]$.  By definition of $\Phi$, the preimages of $\sigma$ are $\Sigma$ and $T$; one checks that the preimages of $\tau$ under $\Phi$ are sections  $T_{\pm}: [X:Y:Z]=[b\pm 2a:0:1]$ on $\mathcal{Y}$. The even eight is formed by two disjoint components in each reducible fiber of type $A_3$.  The situation must be the same for all four reducible fibers. Because we have four sections $\Sigma, T, T_{\pm}$ on $Y$ which are the preimages of $\sigma$ and $\tau$, the components that are met by $\sigma$ in $\tau$ in each $A_3$ fiber on $\mathcal{X}$ must be disjoint and not part of the even eight. The complimentary rational components form the even eight. It follows from (3) in Theorem~\ref{thm} that $\pi_{\mathcal{Z}}$ has no sections. One checks that the preimages of $\sigma$ is the bisection; one checks that the preimages of $\tau$ under $\Phi$ are sections  $T_{\pm}: [X:Y:Z]=[b\pm 2a:0:1]$ on $\mathcal{Y}$.  Therefore, both components in each reducible fiber on $\mathcal{X}$ that are met by $\sigma$ and $\tau$ must be part of the ramification locus.  
\end{proof}
\noindent
The different even eights chosen in $\mathcal{X}$ as the branch locus to obtain the double covers $\mathcal{Y}$ (green) and $\mathcal{Z}$ (yellow) in Proposition~\ref{prop_EEP_I4} are shown in the following diagram:
\begin{center}
\includegraphics[scale=\MyScalePic]{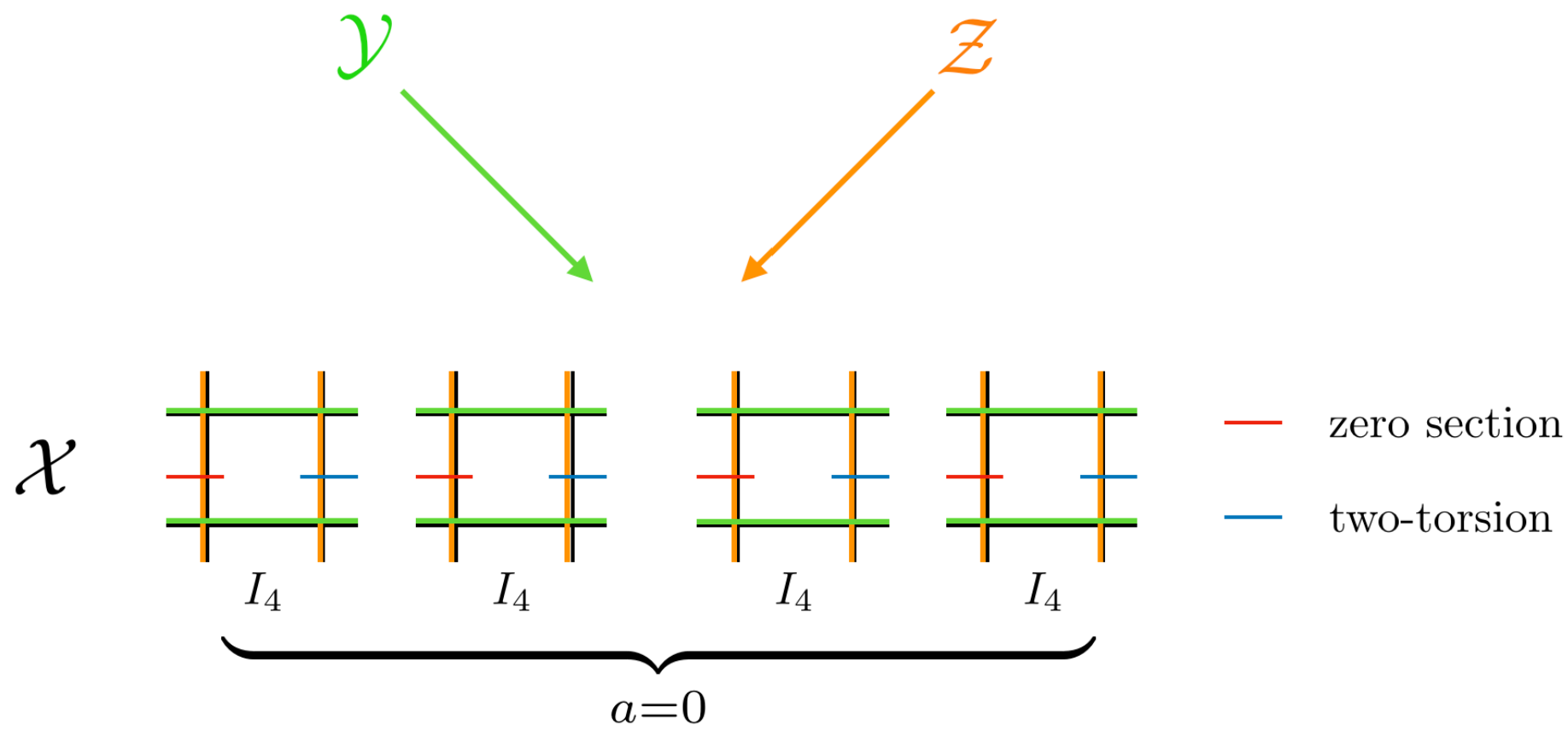}
\end{center}
\subsection{Four $I_0^*$ fibers in Picard rank $\rho=18$}
We assume that $c= a$ and $b=2\beta a$ in Theorem~\ref{thm} where $a$ is homogeneous of degree four in $\mathbb{C}(t_0,t_1)$ with no repeated root and $\mathbb{C} \ni \beta \not = \pm 1$.  We obtain projective models for the K3 surfaces $\mathcal{X}, \mathcal{Y}, \mathcal{Z}$ from the equations
\begin{equation}
\scalemath{\MyScaleMedium}{
\begin{array}{lll}
 \mathcal{Y}_{[t_0:t_1]}:& Y^2 Z  = X  \Big(X - 2(\beta-1)aZ\Big) \Big(X - 2(\beta+1)aZ\Big), & \Delta_{\mathcal{Y}}= \Delta_{\mathcal{Z}},\\[0.5em]
  \mathcal{Z}_{[t_0:t_1]}:& V^2  =  a  U^4 + 2 \beta a U^2W^2 + a W^4, & \Delta_{\mathcal{Z}}=2^8 a^6(\beta^2-1)^2, \\[0.5em]
 \mathcal{X}_{[t_0:t_1]}:& y^2 z  = x \big( x^2 + 2 \beta a x z+  a^2 z^2\big), & \Delta_{\mathcal{X}}=2^4 a^6(\beta^2-1).
\end{array}}
\end{equation} 
In this case, condition~(\ref{condition}) is satisfied, and  the genus-one fibration $\pi_{\mathcal{Z}}$ has sections. The elliptic fibrations with section $\pi_{\mathcal{X}}$ and $\pi_{\mathcal{Y}}$ have the singular fibers $4 I_0^*$ and $\operatorname{MW}(\pi_{\mathcal{X}})=\operatorname{MW}(\pi_{\mathcal{Y}})=(\mathbb{Z}/2\mathbb{Z})^2$. Accordingly, we have $\rho_{\mathcal{X}}=\rho_{\mathcal{Y}}=\rho_{\mathcal{Z}}=18$.  We have the following:
\begin{proposition}
\label{prop_EEP_D4}
In the situation described above, the map $\Phi:\mathcal{Y} \dasharrow \mathcal{X}$ is the double cover branched along the even eight on $\mathcal{X}$ given by the non-central components of the reducible fibers of type $D_4$ over $a=0$ not meeting the sections $\sigma$ or $\tau$. Similarly, the map $\Psi: \mathcal{Z} \dasharrow \mathcal{X}$ is the double cover branched along the even eight on $\mathcal{X}$ given by the components meeting sections $\sigma$ or $\tau$.
\end{proposition}
\begin{proof}
On $\mathcal{X}$ there are sections $\sigma: [x:y:z]=[0:1:0]$, $\tau: [x:y:z]=[0:0:1]$, and $\tau_{\pm}: [x:y:z]=[(-\beta\pm \sqrt{\beta^2-1})a:0:1]$. In this case, condition~(\ref{condition}) is satisfied with $g=0$, $f^2=\beta-\sqrt{\beta^2-1}\in \mathbb{C}$, $e=2\sqrt{\beta^2-1}\,a$, and the K3 surface $\mathcal{Z}$ has four sections. The four sections are given by $[U:V:W]=[\pm U_0:0:1]$ with $U_{\pm}=(-\beta\pm \sqrt{\beta^2-1})^{1/2}$. By definition of $\Psi$, the preimages of $\sigma$ and $\tau$ are the bisections $\kappa: [U:V:W]=[1:\pm\sqrt{a}:0]$ and $\upsilon: [U:V:W]=[0:\pm \sqrt{a}:1]$, respectively.  Each preimage of $\tau_{\pm}$ consists of the two sections $[U:V:W]=[\pm U_0:0:1]$. By definition of $\Phi$, the preimages of $\sigma$ are $\Sigma$ and $T$ on $\mathcal{Y}$; one checks that the preimage of $\tau$ are the sections $[X:Y:Z]=[2(\beta\pm1)a:0:1]$. The even eight consists of two disjoint components in each reducible fiber of type $D_4$. The situation is the same for all four reducible fibers. In each fiber of type $D_4$, each non-central component is met by exactly one two-torsion section.  Because we have four sections $\Sigma, T, T_{\pm}$ on $\mathcal{Y}$ which are the preimages of $\sigma$ and $\tau$, the components that are met by $\sigma$ in $\tau$ in each $D_4$ fiber on $\mathcal{X}$ must not part of the ramification locus. The complimentary rational components form the even eight. For $\mathcal{Z}$ it follows that the four sections are the preimages of the other two sections $\tau_{\pm}$ on $\mathcal{X}$. Accordingly, the even eight is formed by the non-central components of the fibers of type $D_4$ met by $\sigma$ and $\tau$. 
\end{proof}
\noindent
The different even eights chosen in $\mathcal{X}$ as the branch locus to obtain the double covers $\mathcal{Y}$ (green) and $\mathcal{Z}$ (yellow) in Proposition~\ref{prop_EEP_D4} are shown in the following diagram:
\begin{center}
\includegraphics[scale=\MyScalePic]{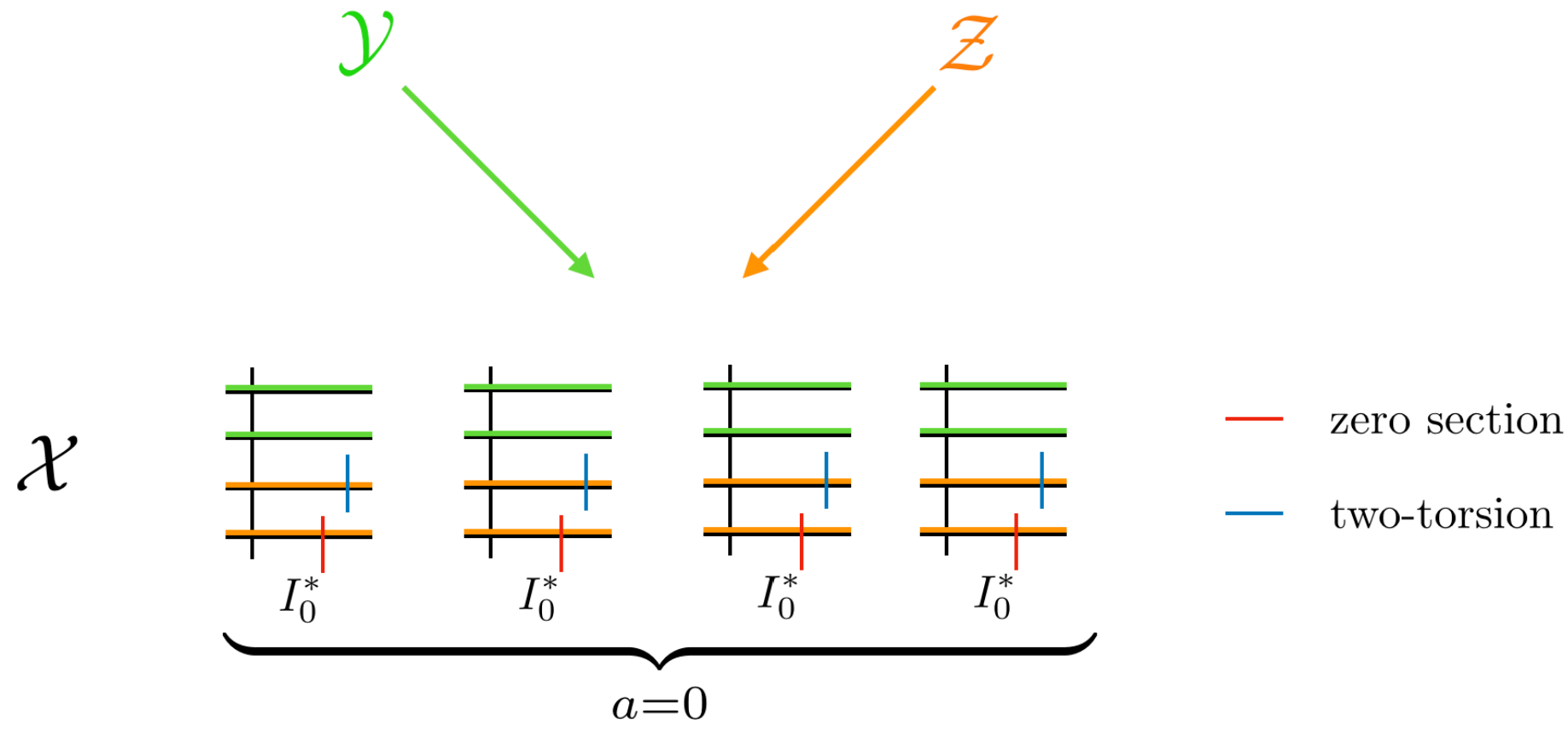}
\end{center}
\section{Examples with Picard numbers $16$ and $17$}
We analyze the situation in which fibers of Kodaira type $I_2$  in the generic fibration with eight $I_2$ and eight $I_1$ coalesce and form configurations with noticeable geometric meaning, namely the so-called double sextic and Kummer configuration. These configurations occur for $k=2$ in Theorem~\ref{thm} and Proposition~\ref{prop_K3} in Picard rank $16$ and $17$. In these cases, our construction in Theorem~\ref{thm} identifies and relates elliptic fibrations, lattice polarizations, and moduli of normal forms for these special geometries. 
\subsection{Kummer surfaces in Picard rank $17$}
We specialize Theorem~\ref{thm} and assume $a=c= \alpha\rho$ and $b=2\beta\rho$ in Theorem~\ref{thm} where $\rho$ is a homogeneous polynomial of degree three in $\mathbb{C}(t_0,t_1)$, $\rho$ has no repeated roots, $\alpha, \beta$ are homogeneous of degree one with no common factor with $\rho$ or each other.  We obtain projective models for K3 surfaces $\mathcal{X}, \mathcal{Y}, \mathcal{Z}$ from the equations
\begin{equation}
\label{eqns_6lines}
\scalemath{\MyScaleMedium}{
\begin{array}{lll}
 \mathcal{Y}_{[t_0:t_1]}:& Y^2 Z  = X \Big(X - 2(\beta-\alpha) \rho Z \Big)  \Big(X - 2(\beta+\alpha) \rho Z\Big), & \Delta_{\mathcal{Y}}=\Delta_{\mathcal{Z}},\\[0.5em]
 \mathcal{Z}_{[t_0:t_1]}:& V^2  =  \alpha\rho U^4 + 2\beta\rho U^2W^2 + \alpha\rho W^4, & \Delta_{\mathcal{Z}}=2^8 \alpha^2 \rho^6 (\beta^2-\alpha^2)^2, \\[0.5em]
 \mathcal{X}_{[t_0:t_1]}:& y^2 z  = x \big( x^2 + 2\beta \rho  x z+  \rho^2 \alpha^2\big), &  \Delta_{\mathcal{X}}=2^4 \alpha^4 \rho^6(\beta^2-\alpha^2) .
\end{array}}
\end{equation}
Condition~(\ref{condition}) is not satisfied, and the genus-one fibration $\pi_{\mathcal{Z}}$ has no sections. The elliptic fibrations with section $\pi_{\mathcal{X}}$ and $\pi_{\mathcal{Y}}$ have the singular fibers $3 I_0^* + I_4 + 2 I_1$ and $\operatorname{MW}(\pi_{\mathcal{X}})=\mathbb{Z}/2\mathbb{Z}$, and $3 I_0^* + 3 I_2$ and $\operatorname{MW}(\pi_{\mathcal{Y}})=(\mathbb{Z}/2\mathbb{Z})^2$, respectively. Accordingly, we have $\rho_{\mathcal{X}} = \rho_{\mathcal{Y}}= \rho_{\mathcal{Z}} = 17$. We have the following:
\begin{proposition}
\label{prop_EEP_Kummer}
In the situation described above,  the map $\Phi:\mathcal{Y} \dasharrow \mathcal{X}$ is the double cover branched along the even eight on $\mathcal{X}$ given by the non-central  components of the reducible fibers of type $D_4$ over $\rho=0$ and of the reducible fiber of type $A_3$ over $a=0$ not meeting the sections $\sigma$ or $\tau$. Similarly, the map $\Psi: \mathcal{Z} \dasharrow \mathcal{X}$ is the double cover branched along the even eight on $\mathcal{X}$ given by the components meeting sections $\sigma$ or $\tau$.
\end{proposition}
\begin{proof}
The proof is obtained by combining the proofs of Propositions~\ref{prop_EEP_I4} and \ref{prop_EEP_D4}. 
\end{proof}
\noindent
The different even eights chosen in $\mathcal{X}$ as the branch locus to obtain the double covers $\mathcal{Y}$ (green) and $\mathcal{Z}$ (yellow) in Proposition~\ref{prop_EEP_Kummer} are shown in the following diagram:
\begin{center}
\includegraphics[scale=\MyScalePic]{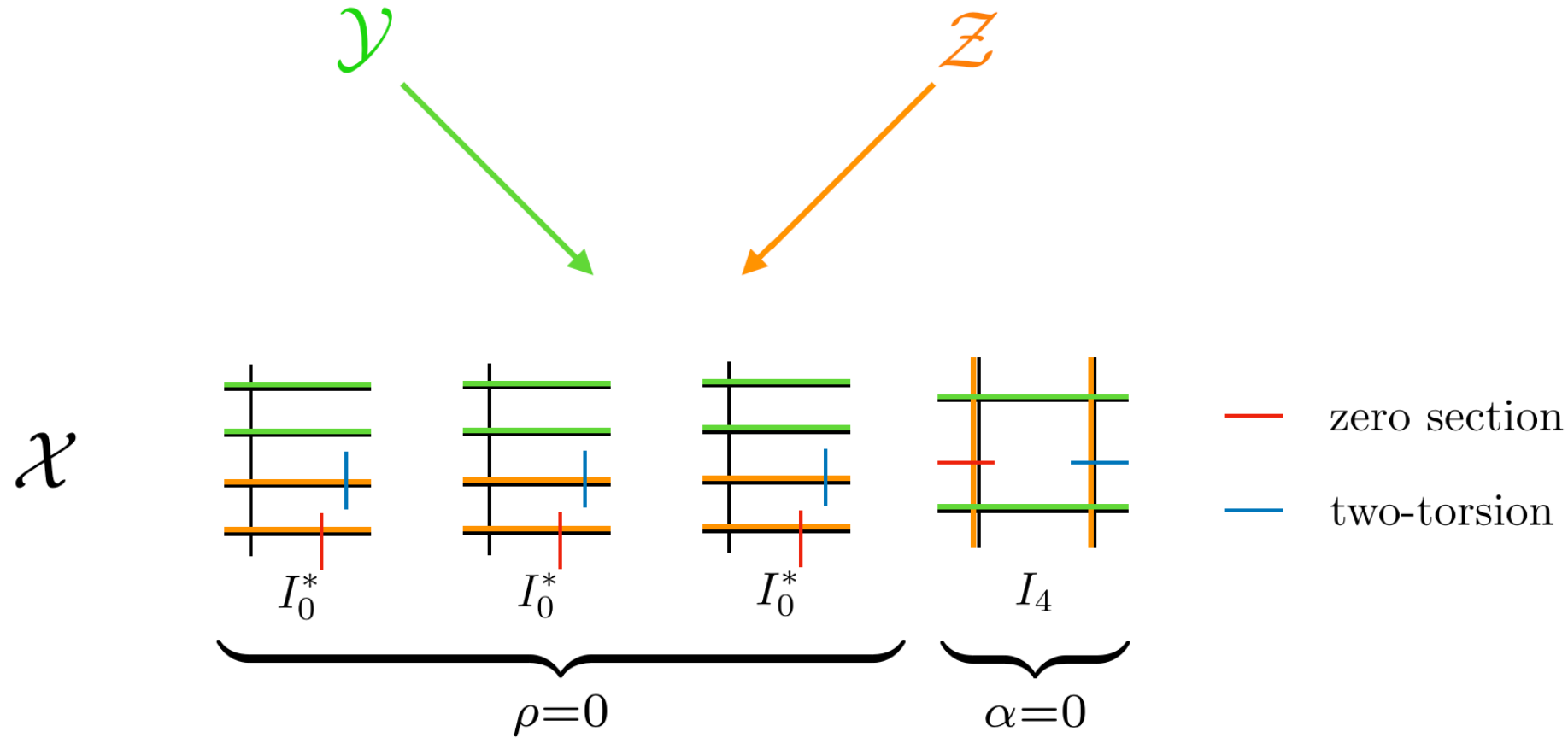}
\end{center}
\par For the Kummer surface $\operatorname{Kum}(\operatorname{Jac} \mathcal{D})$ of the Jacobian variety of a generic smooth genus-two curve for a generic curve  $\mathcal{D}$ of genus two, there are always two sets of sixteen $(-2)$-curves, called nodes and tropes, which are either the exceptional divisors corresponding to blow-up of the 16 two-torsion points $p_0, p_{ij}$ for $1 \le i < j \le6$ or they arise as tropes $T_i, T_{ij6}$ with $1\le i < j <6$ from the embedding of the symmetric theta divisors. These two sets of smooth rational curves have a rich symmetry, the so-called $16_6$-configuration where each node intersects exactly six tropes and vice versa \cite{MR1097176}.   It is now easy to prove the following:
\begin{lemma}
\label{prop_Kummers}
The K3 surface $\mathcal{X}$  is the Kummer surface $\operatorname{Kum}(\operatorname{Jac} \mathcal{D})$ of the principally polarized abelian surface $\operatorname{Jac}(\mathcal{D})$, i.e., the Jacobian variety of a generic smooth genus-two curve $\mathcal{D}$.
\end{lemma}
\begin{proof}
All inequivalent elliptic fibrations on a generic Kummer surface where determined explicitly by Kumar in \cite{MR3263663}. In fact, Kumar computed elliptic parameters and Weierstrass equations for all twenty five different fibrations that appear, and analyzed the reducible fibers and Mordell-Weil lattices. Equation~(\ref{eqns_Kummer}) is the Weierstrass model of the  elliptic fibration {\tt (7)} in the list of all possible elliptic fibrations in~\cite[Thm.~2]{MR3263663}.
\end{proof}
Mehran proved in \cite{MR2804549} that there are 15 distinct isomorphism classes of rational double covers of $\operatorname{Kum}(\operatorname{Jac} \mathcal{D})$ and computed the fifteen even eights consisting of exceptional curves (up to taking complements) on the Kummer surface $\operatorname{Kum}(\operatorname{Jac} \mathcal{D})$ that give rise to all distinct 15 isomorphism classes of rational double  covers \cite[Prop.~4.2]{MR2804549}. Each even eight is enumerated by a node $p_{ij}$ and given as a sum
\[
 \Delta_{ij} = P_{1i} + \dots + \widehat{P_{ij}} + \dots + P_{i6} + P_{1j} + \dots + \widehat{P_{ij}} + \dots + P_{j6} \;,
 \]
where $P_{11}=0$, and $P_{ij}$ are the exceptional divisors obtained by resolving the nodes $p_{ij}$ and the hat indicates divisors that are not part of the even eight. Moreover, Mehran proved that every rational map $\Psi_{\Delta}: \operatorname{Kum}(\mathbf{B}) \dashrightarrow \operatorname{Kum}(\operatorname{Jac} \mathcal{D})$ from a $(1,2)$-polarized to a principally polarized Kummer surface is induced by an isogeny $\psi_{\Delta}:  \mathbf{B} \to \operatorname{Jac}(\mathcal{D})$ of abelian surfaces of  degree two~\cite{MR2804549}, and that all inequivalent $(1,2)$-polarized abelian surfaces $\mathbf{B}$ can be obtained in this way. We have the following:
\begin{corollary}
\label{Kum_B_12}
$\mathcal{Z}$  is the Kummer surface $\operatorname{Kum}(\mathbf{B}_{56})$ of the $(1,2)$-polarized abelian surface $\mathbf{B}_{56}$ that covers $\psi_{56}: \mathbf{B}_{56} \to \operatorname{Jac} (\mathcal{D})$ by an isogeny of degree two such that the induced rational map $\Psi_{56}: \mathcal{Z} = \operatorname{Kum}(\mathbf{B}_{56}) \dashrightarrow \mathcal{Y}=\operatorname{Kum}(\operatorname{Jac} \mathcal{D})$ associated with the even eight $\Delta_{56}$ realizes the double cover $\Psi: \mathcal{Z} \dashrightarrow \mathcal{X}$ in Proposition~\ref{prop_EEP_Kummer}.
\end{corollary}
\begin{proof}
The sum of the components in the even eight that forms the ramification locus of $\Psi$ in Proposition~\ref{prop_EEP_Kummer} is given by
 \[
 \Delta_{56} = P_{15} +  P_{25} + P_{35} + P_{45} + P_{16} + P_{26}+ P_{36} + P_{46}  .
\] 
Moreover, the zero section $\sigma$ and two-torsion section $\tau$ correspond to the tropes $T_6$ and $T_5$, respectively. It follows that the even eight on $\mathcal{X}$ composed of the non-central components of the reducible fiber of type $A_3$ over $a=0$ and of the reducible fiber of type $D_4$ over $\rho=0$ meeting the sections $\sigma$ or $\tau$ or, equivalently, the tropes $T_6$ and $T_5$ form the even eight $ \Delta_{56}$. But this is precisely the even eight forming the branch locus of the map $\Psi$. The result then follows from \cite[Prop.~5.1]{MR2804549}. 
\end{proof}
\noindent
Labels of the components of the reducible fibers of $\mathcal{X}$ in terms of nodes and tropes of the Kummer geometry are shown in the following diagram:
\begin{center}
\includegraphics[scale=\MyScalePic]{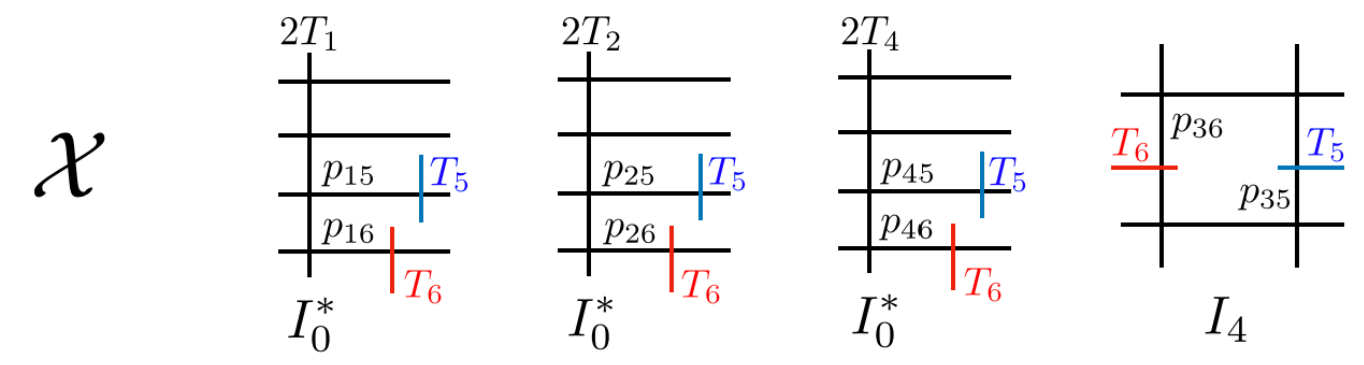}
\end{center}
We have the following:
\begin{proposition}
\begin{enumerate}
\item[]
\item The K3 surfaces $\mathcal{Y}$ and $\mathcal{Z}$ have transcendental lattices $T_{\mathcal{Y}} \cong H(2) \oplus H(2) \oplus \langle -2 \rangle$ and
$T_{\mathcal{Z}} \cong H(2) \oplus H(2) \oplus \langle -8 \rangle$, respectively.
\item There is a one-to-one correspondence between the isomorphism classes of K3 surfaces admitting a rational map of degree two into $\operatorname{Kum}(\operatorname{Jac} \mathcal{D})$ with transcendental lattice either $H(2) \oplus H(2) \oplus \langle -2 \rangle$ or $H(2) \oplus H(2) \oplus \langle -8 \rangle$.
\end{enumerate}
\end{proposition}
\begin{proof}
(1) The K3 surface $\mathcal{Z}$  is the Kummer surface $\operatorname{Kum}(\mathbf{B}_{56})$ obtained by taking the double cover branched along  an even eight of exceptional curves where the polarization of $\mathbf{B}_{56}$ is of type $(1, 2)$, i.e., $T_{\mathbf{B}_{56}} = H \oplus H \oplus \langle -4 \rangle$. There is a Hodge isometry $T_{\mathcal{Z}} = T_{\mathbf{B}_{56}}(2)$. The transcendental lattice of $\mathcal{Y}$ was derived in \cite{Clingher:2017aa}.
\par (2) Mehran proved in \cite{MR2306633} that there are 15 isomorphism classes (each) of K3 surfaces admitting a rational map of degree two into $\operatorname{Kum}(\operatorname{Jac} \mathcal{D})$ with transcendental lattice either $H(2) \oplus H(2) \oplus \langle -2 \rangle$ or $H(2) \oplus H(2) \oplus \langle -8 \rangle$.  These even eights are exactly the ramification loci of the maps $\Phi:\mathcal{Y} \dasharrow \mathcal{X}$ and $\Psi: \mathcal{Z} \dasharrow \mathcal{X}$, respectively. For each even eight of nodes $\Delta_{ij}$ on the Kummer surface $\mathcal{X}\cong \operatorname{Kum}(\operatorname{Jac} \mathcal{D})$, a K3 surface $\mathcal{Z}$ with transcendental lattice $H(2) \oplus H(2) \oplus \langle -8 \rangle$ is obtained by the double cover $\Psi$ in Proposition~\ref{prop_EEP_Kummer}. A complimentary even eight for $\Delta_{ij}$ is then obtained within the components of the reducible fibers of the particular fibration $\pi_{\mathcal{X}}$ used in Proposition~\ref{prop_EEP_Kummer}. And the double cover $\Phi$ is a K3 surface with transcendental lattice $H(2) \oplus H(2) \oplus \langle -2 \rangle$. The correspondence between the two even eights in Proposition~\ref{prop_EEP_Kummer} is one-to-one.
\end{proof}
\subsection{Configuration of six lines in Picard rank $16$}
\label{chap:sixlines}
We assume that $a=c= \alpha\rho$ and $b=2\beta\rho$ in Theorem~\ref{thm} where $\alpha, \beta, \rho$ are homogeneous of degree two in $\mathbb{C}(t_0,t_1)$, and $\alpha, \beta, \rho$ have no repeated roots or common factor. We obtain projective models for K3 surfaces $\mathcal{X}, \mathcal{Y}, \mathcal{Z}$ from the equations
\begin{equation}
\label{eqns_Kummer}
\scalemath{\MyScaleMedium}{
\begin{array}{lll}
 \mathcal{Y}_{[t_0:t_1]}:& Y^2 Z  = X \Big(X - 2(\beta-\alpha) \rho Z \Big)  \Big(X - 2(\beta+\alpha) \rho Z\Big), & \Delta_{\mathcal{Y}}=\Delta_{\mathcal{Z}},\\[0.5em]
 \mathcal{Z}_{[t_0:t_1]}:& V^2  =  \alpha\rho U^4 + 2\beta\rho U^2W^2 + \alpha\rho W^4, & \Delta_{\mathcal{Z}}=2^8 \alpha^2 \rho^6 (\beta^2-\alpha^2)^2, \\[0.5em]
 \mathcal{X}_{[t_0:t_1]}:& y^2 z  = x \big( x^2 + 2\beta \rho  x z+  \rho^2 \alpha^2z^2\big), &  \Delta_{\mathcal{X}}=2^4 \alpha^4 \rho^6 (\beta^2-\alpha^2).
\end{array}}
\end{equation} 
Condition~(\ref{condition}) is not satisfied, and the genus-one fibration $\pi_{\mathcal{Z}}$ has no sections. The elliptic fibrations with section $\pi_{\mathcal{X}}$ and $\pi_{\mathcal{Y}}$ have the singular fibers $2 I_0^* + 2 I_4 + 4 I_1$ and $\operatorname{MW}(\pi_{\mathcal{X}})=\mathbb{Z}/2\mathbb{Z}$, and $2 I_0^* + 6 I_2$ and $\operatorname{MW}(\pi_{\mathcal{Y}})=(\mathbb{Z}/2\mathbb{Z})^2$, respectively. Accordingly, we have $\rho_{\mathcal{X}} = \rho_{\mathcal{Y}} = \rho_{\mathcal{Z}} = 16$. We have the following:
\begin{proposition}
\label{prop_EEP_6lines}
In the situation described above, the map $\Phi:\mathcal{Y} \dasharrow \mathcal{X}$ is the double cover branched along the even eight on $\mathcal{X}$ given by the non-central components of the reducible fibers of type $D_4$ over $\rho=0$ and of the reducible fibers of type $A_3$ over $a=0$ not meeting the sections $\sigma$ or $\tau$. Similarly, the map $\Psi: \mathcal{Z} \dasharrow \mathcal{X}$ is the double cover branched along the even eight on $\mathcal{X}$ given by the components meeting sections $\sigma$ or $\tau$.
\end{proposition}
\begin{proof}
The proof is simply a combination of the proofs of Propositions~\ref{prop_EEP_I4} and \ref{prop_EEP_D4}. The situation is summarized in the following diagram:
\end{proof}
\par We introduce a fourth K3 surface as the genus-one fibration $\pi_{\mathcal{Z}'}: \mathcal{Z}' \to \mathbb{P}^1$ using coordinates $[U':V':W']\in \mathbb{P}(1,2,1)$ for the fiber and the following equation
\begin{equation}
\label{eqns_Kummer_b}
 \mathcal{Z}'_{[t_0:t_1]}:\; V^{\prime 2}  =  \rho U^{\prime 4} + 2\beta\rho U^{\prime 2} W^{\prime 4}  + \alpha^2\rho W^{\prime 4} , \;
  \Delta_{\tilde{\mathcal{Z}}}=2^8 \alpha^2 \rho^6 (\beta^2-\alpha^2)^2.
\end{equation} 
As before, for generic polynomials $\pi_{\mathcal{Z}'}: \mathcal{Z}' \to \mathbb{P}^1$ does not admit a section. A rational map $\Psi': \mathcal{Z}' \dasharrow \mathcal{X}$ of degree two is given by
\begin{equation}
\label{tildeZ}
\begin{split}
  [x:y:z] & =\Psi'([U':V':W'])= [\rho U^{\prime 2} W': \rho U' V' : W^{\prime 3} ].
\end{split}
\end{equation} 
We have the following:
\begin{corollary}
\label{prop_EEP_6lines_b}
In the situation described above, the map $\Psi': \mathcal{Z}' \dasharrow \mathcal{X}$ is the double cover branched along the even eight on $\mathcal{X}$ given by the components of the reducible fibers of type $A_3$ over $a=0$ not meeting the sections $\sigma$ or $\tau$ and of the reducible fibers of type $D_4$ over $\rho=0$ meeting the sections $\sigma$ or $\tau$. 
\end{corollary}
\begin{proof}
The proof is obtained by combining the proofs of Propositions~\ref{prop_EEP_I4} and \ref{prop_EEP_D4}. 
\end{proof}
\noindent
The different even eights chosen in $\mathcal{X}$ as the branch locus to obtain the double covers $\mathcal{Y}$ (green), $\mathcal{Z}$ (yellow) in Proposition~\ref{prop_EEP_6lines}, and $\mathcal{Z}'$ (magenta)  in Corollary~\ref{prop_EEP_6lines_b} are shown in the following diagram:
\begin{center}
\includegraphics[scale=\MyScalePic]{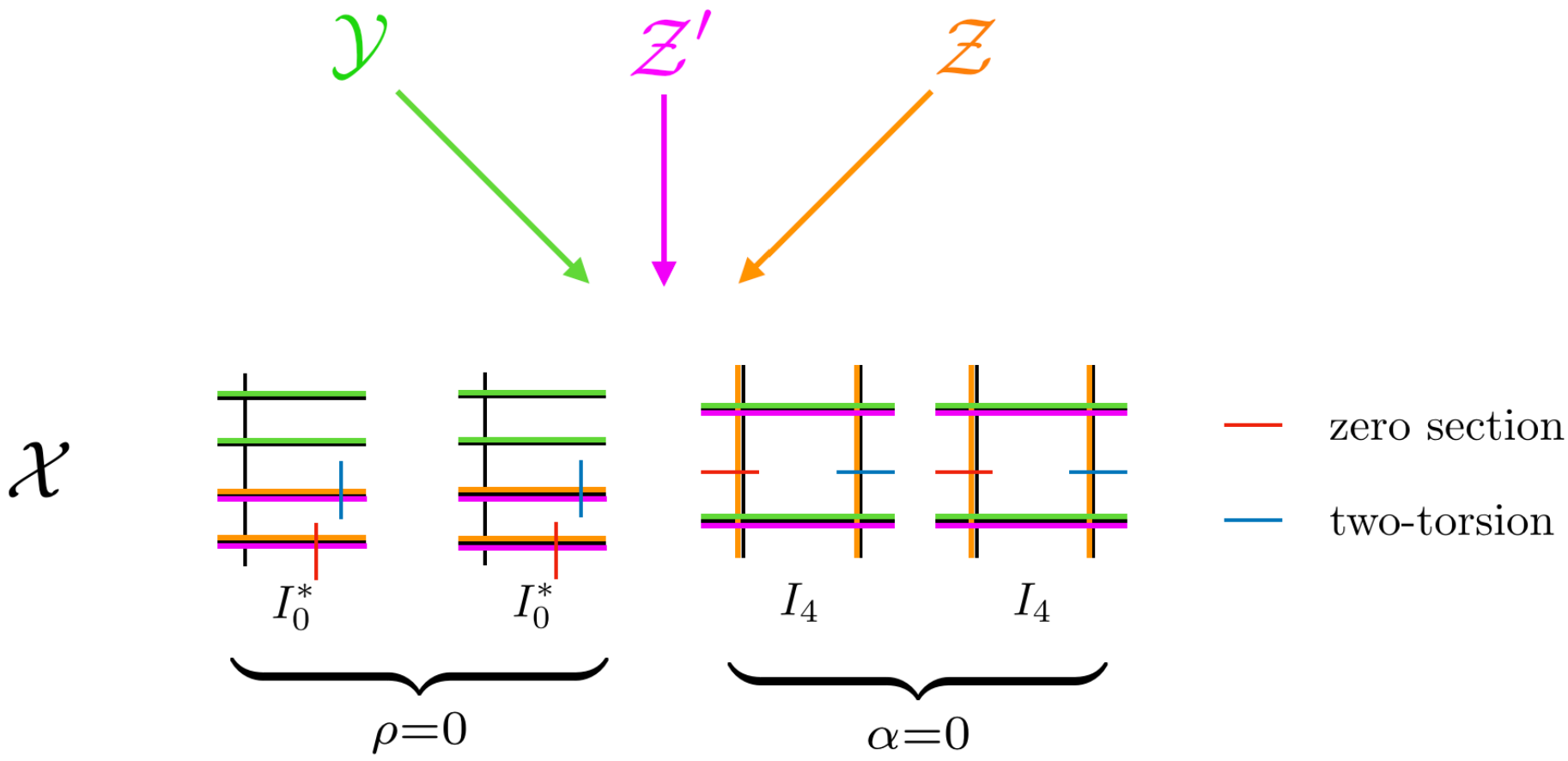}
\end{center}
\par We consider six lines $\ell_i$ in $\mathbb{P}^2$ with $i=1,\dots,6$, no three of which are concurrent. Such a configuration has four moduli which we will denote by $a, b, c, d$. The lines can always be brought into the form
\begin{equation}
 \ell_1=z_1, \ell_2 = z_2, \ell_3= z_3,  \ell_4=z_1 + z_2 + z_3,  \ell_5=a z_1 + b z_2 + z_3, \ell_6=c z_1 + d z_2 + z_3,
\end{equation}
such that the lines intersect as follows:
\begin{equation}
\label{tab:6lines}
\scalemath{\MyScaleTiny}{
\begin{array}{c|c|c|c|c|c|c}
		& \ell_1	& \ell_2	& \ell_3	& \ell_4		& \ell_5			& \ell_6			\\ \hline &&&&&& \\[-0.9em]
\ell_1	& -		& [0:0:1]	& [0:1:0]	& [0:1:-1]		& [0:1:-b]			& [0:1:-d]			\\ &&&&&& \\[-0.9em]
\ell_2	& [0:0:1]	& -		& [1:0:0]	& [1:0:-1]		& [1:0:-a]			& [1:0:-c]			\\ &&&&&& \\[-0.9em]
\ell_3	& [0:1:0]	& [1:0:0]	& -		& [1:-1:0]		& [b:-a:0]			& [d:-c:0]			\\ &&&&&& \\[-0.9em]
\ell_4	& [0:1:-1]	& [1:0:-1]	& [1:-1:0]	& -			& [b-1:1-a:a-b]		& [d-1:1-c:c-d]		\\ &&&&&& \\[-0.9em]
\ell_5	& [0:1:-b]	& [1:0:-a]	& [b:-a:0]	& [b-1:1-a:a-b]	& -				& [b-d:c-a:ad-bc]	\\ &&&&&& \\[-0.9em]
\ell_6	& [0:1:-d]	& [1:0:-c]	& [d:-c:0]	& [d-1:1-c:c-d]	& [b-d:c-a:ad-bc]	& -				\\ &&&&&& \\[-0.9em]
\end{array}}
\end{equation}
The double cover branched along six lines $\ell_i$ in $\mathbb{P}^2$ with $i=1,\dots,6$ is given by
\begin{equation}
\label{six-lines-cover}
z_4^2 = \ell_1\ell_2\ell_3\ell_4\ell_5\ell_6 .
\end{equation}
Setting $z_1=t_0$, $z_2=t_1$, Equation~(\ref{six-lines-cover}) becomes the Weierstrass model for $\mathcal{Y}_{[t_0:t_1]}$ in Equation~(\ref{eqns_Kummer}) with
\begin{equation}
\label{coeffs}
\begin{split}
\rho & = t_0 t_1,\\
 2(\beta-\alpha) & = (at_0 +b t_1) ((c-1) t_0 + (d-1) t_1) ,\\
 2(\beta+\alpha) & = (c t_0 +d t_1) ((a-1) t_0 + (b-1) t_1).
\end{split}
\end{equation}
Conversely, in every Weierstrass model for $\mathcal{Y}$ in Equation~(\ref{eqns_Kummer}) the coordinates $t_0, t_1$ on the base curve can be chosen such that Equations~(\ref{coeffs}) hold. We have the following lemma:
\begin{proposition}
\label{lem_6lines}
For generic $a, b, c, d$ we have the following:
\begin{enumerate}
\item The K3 surface $\mathcal{Y}$ is the double cover of $\mathbb{P}^2$ branched along six lines, no three of which are concurrent.
\item $\mathcal{Y}$  has the transcendental lattice $T_{\mathcal{Y}} \cong H(2) \oplus H(2) \oplus \langle -2 \rangle^{\oplus 2}$.
\end{enumerate}
\end{proposition}
\begin{proof}
(1) is already established. (2) It was proven in~\cite{MR1877757} that the transcendental lattice of $z_4^2 = \ell_1\ell_2\ell_3\ell_4\ell_5\ell_6$ is $H(2) \oplus H(2) \oplus \langle -2 \rangle^{\oplus 2}$.
\end{proof}
\begin{remark}
In \cite{MR1871336} a six-parameter family of K3 surfaces of Picard rank 14 with transcendental lattice $H(2)^{\oplus 2} \oplus \langle -2 \rangle^{\oplus 4}$ was described.  The general member of this family is exhibited as a double cover of $\mathbb{P}^1\times\mathbb{P}^1$ branched along the union of four curves of bi-degree $(1,1)$. The standard form used in \cite{MR1871336} was given in affine coordinates by
\begin{equation}
\label{KSTT}
 \rho^2 = \prod_{k=1}^4 \left( \alpha^{(k)} \xi \eta + \beta^{(k)} \xi + \gamma^{(k)} \eta + \delta^{(k)} \right) 
\end{equation} 
with matrices 
\begin{equation}
 M_k = \left( \begin{array}{cc}  \alpha^{(k)} & \beta^{(k)} \\ \gamma^{(k)} & \delta^{(k)} \end{array} \right) 
\end{equation}
for $k=1,\dots,4$. Setting $t_0=\xi$, $t_1=1$, $Z=1$, and
\begin{equation}
 X = \xi (a\xi+b)(c\xi+d)\eta, \quad Y =\xi (a\xi+b)(c\xi+d)\rho
\end{equation}
in Equation~(\ref{eqns_Kummer}) defining $\mathcal{Y}$, we obtain Equation~(\ref{KSTT}) with matrices
\begin{equation}
\scalemath{\MyScaleSmall}{
 M_1 = \left( \begin{array}{cc} 0 & 1 \\ 0 & 0  \end{array} \right), \;
 M_2 = \left( \begin{array}{cc} 0 & 0 \\ 1 & 0  \end{array} \right), \;
 M_3 = \left( \begin{array}{cc} a & b \\ 1-a & 1-b  \end{array} \right), \;
 M_4 = \left( \begin{array}{cc} c & d \\ 1-c & 1-d  \end{array} \right),
}
\end{equation}
and determinants $\det M_1=\det M_2=0$, $\det M_3=a-b$ and $\det M_2=c-d$.
\end{remark}
We also have the following:
\begin{proposition}
\label{Gamma-fibration}
The K3 surface $\mathcal{Z}$ admits a second elliptic fibration $\check{\pi}_{\mathcal{Z}}: \mathcal{Z} \to \mathbb{P}^1$ with section, with singular fibers $8I_2 + 8 I_1$, and a Mordell-Weil group such that $\operatorname{MW}(\check{\pi}_{\mathcal{Z}})_{\mathrm{tor}}=\mathbb{Z}/2\mathbb{Z}$.
\end{proposition}
\begin{proof}
The key observation is that \emph{all} coefficients of $U^4, U^2W^2, W^4$ in Equation~(\ref{eqns_Kummer}) are homogeneous polynomials in $\mathbb{C}(t_0, t_1)$ of degree four. We interchange the roles of base and fiber coordinates. Since $\rho=t_0t_1$ the equation is in Weierstrass form with rational two-torsion. The proof then follows by direct computation using (7) in Proposition~\ref{prop_K3}.
\end{proof}
\par We consider the double cover $f: \mathbb{P}^1_{[s_0:s_1]} \to \mathbb{P}^1_{[t_0:t_1]}$ given by $[s_0:s_1] \mapsto [t_0:t_1] =[s^2_0:s^2_1]$ invariant under the involution $[s_0:s_1]\mapsto [s_0:-s_1]$. The polynomials $\alpha, \beta$ then become homogeneous polynomials of degree four in $\mathbb{C}(s_0,s_1)$, invariant under the involution, with no repeated roots or common factors with each other or $s_0s_1$.  We obtain another K3 surface as the elliptically fibered surface $\pi_{\tilde{\mathcal{Y}}}: \tilde{\mathcal{Y}} \to \mathbb{P}^1$ with section using coordinates $[\tilde{X}:\tilde{Y}:\tilde{Z}]\in \mathbb{P}^2$ for the fiber and $[s_0:s_1]\in \mathbb{P}^1$ for the base from the equation
\begin{equation}
\label{eqns_Kummer_up}
 \tilde{\mathcal{Y}}_{[s_0:s_1]}:\; \tilde{Y}^2 \tilde{Z}  = \tilde{X}  \Big(\tilde{X }- 2(\beta-\alpha)  \tilde{Z}\Big)  \Big(\tilde{X} - 2(\beta+\alpha) \tilde{Z}\Big), \;\Delta_{\tilde{\mathcal{Y}}}=2^8 \alpha^2 (\beta^2-\alpha^2)^2.
\end{equation} 
Alternative projective models for the K3 surfaces $\mathcal{Y}$ and $\tilde{\mathcal{Y}}$ are given in terms of the coordinates $[u:v:w], [\tilde{u}:\tilde{v}:\tilde{w}]\in \mathbb{P}(1,2,1)$ by the equations
\begin{equation}
\label{eqns_Kummer_Ytb}
\begin{array}{lll}
 \tilde{\mathcal{Y}}_{[s_0:s_1]}:& \tilde{v}^2  =  \tilde{u}^4 + 2\beta \tilde{u}^2\tilde{w}^2 + \alpha^2 \tilde{w}^4, & \Delta_{\tilde{\mathcal{Y}}}=2^8 \alpha^2 (\beta^2-\alpha^2)^2,\\[0.5em]
 \mathcal{Y}_{[t_0:t_1]}:& v^2  =  u^4 + 2t_0 t_1 \beta u^2w^2 + t_0^2t_1^2 \alpha^2 w^4, \quad & \Delta_{\mathcal{Y}}=2^8 \alpha^2 (\beta^2-\alpha^2)^2 t_0^6t_1^6.
\end{array}
\end{equation} 
By Lemma~\ref{lemma_reduction}, Equations~(\ref{eqns_Kummer_Ytb}) can be transformed back over $k(s_0,s_1)$ or $k(t_0,t_1)$ into Equations~(\ref{eqns_Kummer}) and~(\ref{eqns_Kummer_up}), respectively. We have the following:
\begin{proposition}
\label{diamond}
In the situation described above, the maps $F: \tilde{\mathcal{Y}}\dasharrow \mathcal{Y}$ and $F': \tilde{\mathcal{Y}}\dasharrow \mathcal{Z}'$ given by
\begin{equation}
\label{maps}
\scalemath{\MyScaleSmall}{
\begin{array}{rcl}
 F:	\Big([s_0:s_1],[\tilde{u}: \tilde{v}: \tilde{w}]\Big)	&  \mapsto 	
 &	\Big([t_0:t_1],[u:v:w]\Big) =\Big([s^2_0:s^2_1],[s_0s_1\tilde{u}:s^2_0s^2_1\tilde{v}:\tilde{w}]\Big), \\[0.5em]
F':	\Big([s_0:s_1],[\tilde{u}: \tilde{v}: \tilde{w}]\Big)	&  \mapsto 	
& 	\Big([t_0:t_1],[U':V':W']\Big) =\Big([s^2_0:s^2_1],[\tilde{u}:s_0s_1\tilde{v}:\tilde{w}]\Big),
\end{array}}
\end{equation} 
are the rational maps of degree two associated with the Nikulin involutions on $\tilde{\mathcal{Y}}$
\begin{equation}
\begin{array}{lrcl}
  \jmath_{\tilde{\mathcal{Y}}}:&  \Big([s_0:s_1],[\tilde{u}: \tilde{v}: \tilde{w}]\Big) & \mapsto & \Big([s_0:-s_1],[-\tilde{u}: \tilde{v}: \tilde{w}]\Big),  \\[0.5em]
  \jmath'_{\tilde{\mathcal{Y}}}:&  \Big([s_0:s_1],[\tilde{u}: \tilde{v}: \tilde{w}]\Big) & \mapsto &  \Big([s_0:-s_1],[\tilde{u}: -\tilde{v}: \tilde{w}]\Big),
\end{array}
\end{equation}
such that $F \circ  \jmath_{\tilde{\mathcal{Y}}}=F$ and $F' \circ  \jmath'_{\tilde{\mathcal{Y}}}=F'$.
The following diagram commutes and all arrows between K3 surfaces are rational maps of degree two associated with Nikulin involutions:
\begin{equation}
\xymatrix{
  & \tilde{\mathcal{Y}}  \ar @(dl,ul) _{\jmath'_{\tilde{\mathcal{Y}}}}  \ar @(dr,ur) ^{\jmath_{\tilde{\mathcal{Y}}}} \ar[d] \ar[dl]_{F} \ar[dr]^{F'}\\
  \mathcal{Y} \ar @(dl,ul) _{\imath_{\mathcal{Y}}} \ar[d]  \ar[dr]^(.35){\Phi} & \mathbb{P}^1_{[s_0:s_1]}  \ar@{->}[dl]^(.35){f}|\hole \ar@{->}[dr]_(.35){f}|\hole  & \mathcal{Z}'  \ar @(dr,ur) ^(0.75){\imath_{\mathcal{Z}'}} \ar[d] \ar[dl]_(.35){\Psi'} \\
   \mathbb{P}^1_{[t_0:t_1]} \ar[dr]_{=} & X  \ar[d] &  \mathbb{P}^1_{[t_0:t_1]} \ar[dl]^{=} \\
  & \mathbb{P}^1_{[t_0:t_1]}
}
\end{equation}
\end{proposition}
\begin{proof} The Nikulin involution $\imath_{\mathcal{Y}}$ on $\mathcal{Y}$ -- acting by fiberwise translation by two-torsion -- such that $\mathcal{X}=\widehat{\mathcal{Y}/\{1,\imath_{\mathcal{Y} }\}}$ is represented by $\imath_{\mathcal{Y}}: [u:v:w] \mapsto [-u:-v:w]$. The rational map in Equation~(\ref{tildeZ}) is invariant under the Nikulin involution $\imath_{\mathcal{Z}'} : [U':V':W'] \mapsto [-U':-V':W']$. Both involutions preserve fibers. One easily checks that the maps in Equation~(\ref{maps}) are rational maps of degree two. The involutions $\jmath_{\tilde{\mathcal{Y}}}$ and $\jmath'_{\tilde{\mathcal{Y}}}$ interchange the sheets of $F$ and $F'$, respectively. Moreover, we check that the pullback by $f$ of the holomorphic one-form $t_1 dt_0-t_0 dt_1$ is $2s_0s_1 (s_1 ds_0 - s_0 ds_1)$. Therefore, the pullback of the two-form $(t_1 dt_0-t_0 dt_1) \wedge (w du - u dw)/v$ on $\mathcal{Y}$ by $F$ equals $2(s_1 ds_0 - s_0 ds_1)\wedge(\tilde{w} d\tilde{u} - \tilde{u} d\tilde{w})/\tilde{v}$. A similar computation applies to $F'$.
\end{proof}
We have the following:
\begin{corollary}
\label{even_eight_for_base_transfo}
In the situation described above, the maps $F: \tilde{\mathcal{Y}} \dasharrow \mathcal{Y}$ and $F': \tilde{\mathcal{Y}} \dasharrow \mathcal{Z}'$ are double cover branched along the even eights on $\mathcal{Y}$ and $\mathcal{Z}'$ given by the non-central components of the reducible fibers of type $D_4$ over $t_1=0$ and $t_0=0$.
 \end{corollary}
\begin{proof}
The proof follows from the fact that the map $f: \mathbb{P}^1_{[s_0:s_1]} \to \mathbb{P}^1_{[t_0:t_1]}$ given by $[s_0:s_1] \mapsto [t_0:t_1] =[s^2_0:s^2_1]$ has ramification points $t_0=0$ and $t_1=0$.
\end{proof}
We have the following:
\begin{proposition}
\label{prop:fib}
For generic $a, b, c, d$ in Equations~(\ref{coeffs}) we have the following:
\begin{enumerate}
\item $\tilde{\mathcal{Y}}$ is the double cover of $\mathbb{P}^1 \times \mathbb{P}^1$ with branch locus the union of four curves of bi-degree $(1, 0)$, $(1, 0)$, $(1, 2)$, $(1, 2)$.
\item $\pi_{\tilde{\mathcal{Y}}}$ has the singular fibers $12 I_2$ and $\operatorname{MW}(\pi_{\tilde{\mathcal{Y}}})=(\mathbb{Z}/2\mathbb{Z})^2+\langle 1 \rangle^{\oplus 2}$.
\end{enumerate}
\end{proposition}
\begin{proof}
(1)  Setting $s_0=\tilde{\xi}$, $s_1=1$, $\tilde{Z}=1$, and
\begin{equation}
 \tilde{X} = (a\tilde{\xi}+b)(c\tilde{\xi}+d)\tilde{\eta}, \quad \tilde{Y} =(a\tilde{\xi}+b)(c\tilde{\xi}+d)\tilde{\rho}
\end{equation}
in Equation~(\ref{eqns_Kummer_up}) defining $\tilde{\mathcal{Y}}$, we obtain the affine model
\begin{equation}
 \tilde{\rho}^2 = \tilde{\eta}\Big((a\tilde{\xi}^2+b)\tilde{\eta}+(1-a)\tilde{\xi}^2+1-b\Big)\Big((c\tilde{\xi}^2+d)\tilde{\eta}+(1-c)\tilde{\xi}^2+1-d\Big).
\end{equation}
\par (2) A theorem by Inose states \cite[Cor.~1.2]{MR0429915} that two K3 surfaces defined over a number field and connected by a rational map of finite degree have the same Picard rank. Since the Picard rank of $\mathcal{Y}$ equals $16$ there must be two non-torsion sections on $\tilde{\mathcal{Y}}$. Two sections can be constructed explicitly as follows: sections $s_1$ and $s_2$ do not intersect the zero section $\Sigma$, and each intersect six neutral and six non-neutral components in the twelve $I_2$ fibers. The inverse Cartan matrix $A_1^*$ of an $A_1$-fibre is $(\frac{1}{2})$ for the non-neutral component. We obtain the following height pairings with $i=1,2$:
 \begin{equation}
   \langle s_i, s_i \rangle = 2 \cdot \chi^{\text{hol}} + 2 \, s_i \circ \Sigma - 6 \cdot \frac{1}{2}  - 6 \cdot 0 = 1,
 \end{equation}  
where $\chi^{\text{hol}}=2$ is the holomorphic Euler characteristic of a K3 surface. The sections $s_1$ and $s_2$ meet  complimentary non-neutral components of $A_1$-fibers, in particular the set of $I_2$-fibers met by $s_1$ and $s_2$ are mapped into each other by the involution $[s_0:s_1]\mapsto[s_0:-s_1]$. Moreover, the  sections $s_1$ and $s_2$ intersect each other twice. Thus, we have
 \begin{equation}
   \langle s_1, s_2 \rangle = \chi^{\text{hol}} + s_1 \circ \Sigma + s_2 \circ \Sigma -  s_1 \circ s_2 - 0 = 0.
 \end{equation}  
\end{proof}
\subsection{Special configurations of six lines}
We examine two specialization of the configuration of six lines in Proposition~\ref{lem_6lines}. We have the following:
\begin{proposition}
\label{special1}
For $abc-abd-acd+bcd+ad-bc =0$ and otherwise generic $a, b, c, d$ in Equations~(\ref{coeffs}), the following holds:
\begin{enumerate}
\item $\mathcal{Y}$ is the double cover of $\mathbb{P}^2$ branched along six lines tangent to a conic.
\item $\mathcal{Y}$ is the Kummer surface of $\operatorname{Jac}(\mathcal{D})$ for a genus-two curve $\mathcal{D}$.
\item $\pi_{\mathcal{Y}}$ has the singular fibers $2I_0^* + 6I_2$ and $\operatorname{MW}(\pi_{\mathcal{Y}})=(\mathbb{Z}/2\mathbb{Z})^2+\langle 1 \rangle$.
\item $\pi_{\tilde{\mathcal{Y}}}$ has the singular fibers $12 I_2$ and $\operatorname{MW}(\pi_{\tilde{\mathcal{Y}}})=(\mathbb{Z}/2\mathbb{Z})^2+\langle 1 \rangle^{\oplus 2} + \langle 2 \rangle$.
\end{enumerate}
\end{proposition}
\begin{proof}
The conic $(\alpha z_1 + (1-\alpha) z_2 + z_3)^2-4\alpha (1-\alpha)z_1z_2=0$ is smooth for $\alpha \not =0,1$. The lines $z_1=0$, $z_2=0$, $z_3=0$, $z_1+z_2+z_3=0$ are tangent to this conic. The line $az_1+bz_2+z_3=0$ is tangent to the conic if and only if $\alpha=a(1-b)/(a-b)$. For this parameter $\alpha$, the line $cz_1+dz_2+z_3=0$ is tangent to the conic if and only if $abc-abd-acd+bcd+ad-bc =0$. Thus, $\mathcal{Y}$ is the double cover of $\mathbb{P}^2$ branched along six lines tangent to a conic, and hence the Kummer surface of  $\operatorname{Jac}(\mathcal{D})$ where the genus-two curve $\mathcal{D}$ is  given by Equation~(\ref{Rosenhain}) (cf.~\!\cite{Clingher:2017aa}). Computing the discriminant it follows that $\mathcal{Y}$ has singular fibers $3I_0^* + 3I_2$. It was proven in \cite{MR3263663} that the fibration on the Kummer surface of  $\operatorname{Jac}(\mathcal{D})$ with these singular fibers has $\operatorname{MW}(\pi_{\mathcal{Y}})=(\mathbb{Z}/2\mathbb{Z})^2$. This proves (1), (2), (3). It was proved in~\cite{Clingher:2017aa} that the even eight on $\mathcal{Y}$ given by the non-central components of the reducible fibers of type $D_4$ over $t_1=0$ and $t_0=0$ is the even eight $\Delta_{56}$ of exceptional curves from nodes. (4) then follows from \cite[Prop.~5.1]{MR2804549}. 
\end{proof}
We have another specialization of the configuration of six lines in Proposition~\ref{lem_6lines}.:
\begin{proposition}
\label{special2}
For $a=b$ or $c=d$ or $ad-bc=0$ or $ad-bc=a-b-c+d$ and otherwise generic $a, b, c, d$ we have the following:
\begin{enumerate}
\item $\pi_{\mathcal{Y}}$ has the singular fibers $3I_0^* + 3I_2$ and $\operatorname{MW}(\pi_{\mathcal{Y}})=(\mathbb{Z}/2\mathbb{Z})^2$.
\item $\mathcal{Y}$ is the double cover branched along six lines, three are coincident in a point.
\item $\pi_{\mathcal{X}}$ has the singular fibers $3I_0^* + I_4+2I_1$ and $\operatorname{MW}(\pi_{\mathcal{Y}})=\mathbb{Z}/2\mathbb{Z}$.
\item $\mathcal{X}$ is the double cover branched along a conic, a line tangent to the conic, and three lines coincident in a point.
\item $\mathcal{X}$ is the Kummer surface of $\operatorname{Jac}(\check{\mathcal{D}})$ for a genus-two curve $\check{\mathcal{D}}$.
\end{enumerate}
\end{proposition}
\begin{proof}
The right hand side of Equation~(\ref{six-lines-cover})  has a discriminant (with respect to $z_3$) given by
\begin{equation}
\begin{split}
\Delta & = z_1^6 z_2^6 (z_1+z_2)^2 (a z_1 +bz_2)^2 (c z_1 + d z_2)^2 ((a-1)z_1+(b-1)z_2)^2\\
& \times  ((c-1)z_1+(d-1)z_2)^2 ((a-c)z_1+(b-d)z_1)^2.
\end{split}
\end{equation}
Exactly three fibers of type $I_2$ coalesce to form a fiber of type $I_0^*$ if and only if one of the given constraints is satisfied. This proves (1) and (2). Plugging Equations~(\ref{coeffs}) into $\Delta_{\mathcal{X}}$ in Equation~(\ref{eqns_Kummer}) proves (3).  In the case any one of the constraints is satisfied it follows that $\alpha$ and $\beta$ in Equation~(\ref{eqns_Kummer}) have a common factor $\kappa$ homogeneous of degree one in $k(t_0,t_1)$. We write $\alpha=\kappa \tilde{\alpha}$ and $\beta=\kappa \tilde{\beta}$. The K3 surface $\mathcal{X}$ in Equation~(\ref{eqns_Kummer}) is equivalent to 
\begin{equation}
\label{conic}
\tilde{y}^2  = t_0 t_ 1 \kappa \, x ( x^2 + 2\tilde{\beta}  x +   \tilde{\alpha}^2).
\end{equation}
The lines $t_0=0$, $t_1=0$, $\kappa=0$ meet in the point $[t_0:t_1:x]=[0:0:1]$. The line $x=0$ is tangent to the conic $x^2 + 2\tilde{\beta}  x +   \tilde{\alpha}^2=0$. This proves (4).  These are the same singular fibers and group of sections as fibration labeled {\tt (4)}  in the list of all possible elliptic fibrations on a generic Kummer surface determined  by Kumar in~\cite[Thm.~2]{MR3263663}. This proves (5).  
\end{proof}
\noindent
The situation in Proposition~\ref{special1} (1) is shown in Figure~\ref{Fig:6Lines1}, Proposition~\ref{special2} (1) and (3) are shown in Figure~\ref{Fig:6Lines2} and~\ref{Fig:6Lines3}.
\begin{figure}[ht]
    \centering
    \begin{subfigure}[b]{0.2\textwidth}
    \scalemath{\MyScalePicSmall}{
	$$
  	\begin{xy}
   	<0cm,0cm>;<1.5cm,0cm>:
    	(.5,0.18);(3.5,0.18)**@{-},
    	(0.5,2);(1.6,0)**@{-},
    	(3.1,2);(2.6,0)**@{-},
    	(0.7,.5);(2,2.5)**@{-},
    	(0.5,1.82);(3.5,1.82)**@{-},
    	(2,2.7);(3.45,0)**@{-},
    	(2,1)*\xycircle(.8,.8){}
  	\end{xy}
  	$$}
	\caption{\label{Fig:6Lines1}}
    \end{subfigure}
    \qquad
    \begin{subfigure}[b]{0.2\textwidth}
	\scalemath{\MyScalePicSmall}{
  	$$
  	\begin{xy}
    	<0cm,0cm>;<1cm,0cm>:
    	(-1,0.5);(5.5,0.5)**@{-},
    	(-1,2.8);(5.5,-0.5)**@{-}, 
    	(-1,2);(5.5,1.1)**@{-},
    	(-1,-0.5);(5,3.5)**@{-},
    	(4,3.5);(4,-0.5)**@{-},
    	(4.35,3.5);(2.2,-0.5)**@{-}
  	\end{xy}
  	$$}
	\caption{\label{Fig:6Lines2}}
    \end{subfigure}
    \qquad  
    \begin{subfigure}[b]{0.2\textwidth}
    \scalemath{\MyScalePicSmall}{
  	$$
  	\begin{xy}
    	<0cm,0cm>;<1cm,0cm>:
    	(-1,0);(5.5,0)**@{-},
    	(-1,-0.5);(5,3.5)**@{-},
    	(4,3.5);(4,-0.5)**@{-},
    	(4.35,3.5);(2.2,-0.5)**@{-},
    	(2,0.9)*\xycircle(2.2,.9){}
  	\end{xy}
  	$$}
	\caption{\label{Fig:6Lines3}}
    \end{subfigure}
    \caption{\label{Fig:6Lines}}
\end{figure}
\subsection{Moduli of Kummer surfaces in Picard rank $17$}
A generic, smooth, genus-two curve $\mathcal{D}$ is given in Rosenhain normal form by
\begin{equation}\label{Rosenhain}
 \mathcal{D}: \quad y^2 = x (x-1) (x-\lambda_1) (x-\lambda_2) (x-\lambda_3) .
\end{equation} 
We introduce pairwise distinct moduli $\mu_1, \mu_2, \mu_3$ that are determined by the moduli of a genus-two curve $\mathcal{D}$ in Equation~(\ref{Rosenhain}) by the equations
\begin{equation}
\label{Rosenhain_roots}
\scalemath{\MyScaleMedium}{
\begin{array}{c}\displaystyle
 \mu_1 = \frac{\lambda_1 + \lambda_2\lambda_3}{L}, \;  \mu_2 = \frac{\lambda_2 + \lambda_1\lambda_3}{L}, \;  \mu_3 = \frac{\lambda_3 + \lambda_1\lambda_2}{L}
 \end{array}}
\end{equation}
with $L^2 = 4 \lambda_1\lambda_2\lambda_3$.  We also introduce the `dual'  set of moduli $\check{\mu}_1, \check{\mu}_2, \check{\mu}_3$ given by
\begin{equation}
\label{dual_moduli}
\scalemath{\MyScaleMedium}{
\begin{array}{c}\displaystyle
\check{\mu}_1 = \frac{2\mu_1-\mu_2-\mu_3}{\mu_2-\mu_3}, \; \check{\mu}_2 = \check{\mu}_1 - \frac{2(\mu_1-\mu_2)(\mu_1-\mu_3)}{(\mu_1+1)(\mu_2-\mu_3)}, \; 
\check{\mu}_3 = \check{\mu}_1 - \frac{2(\mu_1-\mu_2)(\mu_1-\mu_3)}{(\mu_1-1)(\mu_2-\mu_3)}.
 \end{array}}
\end{equation}
It is easy to show that the relation between $\mu_1, \mu_2, \mu_3$ and $\check{\mu}_1, \check{\mu}_2, \check{\mu}_3$ is symmetric. The following was proven in \cite{Clingher:2017aa}:
\begin{proposition}[\cite{Clingher:2017aa}]
The Jacobian of $\check{\mathcal{D}}:=\mathcal{D}(\check{\mu}_1, \check{\mu}_2, \check{\mu}_3)$ is $(2,2)$-isogeneous to the Jacobian variety of $\mathcal{D}:=\mathcal{D}(\mu_1, \mu_2, \mu_3)$. In particular, $\operatorname{Jac}(\check{\mathcal{D}})$ is a principally polarized abelian surface.
\end{proposition}
We now derive explicit Weierstrass models for the Kummer surfaces of $\operatorname{Jac}(\mathcal{D})$ and $\operatorname{Jac}(\check{\mathcal{D}})$ and $\mathbf{B}_{56}$. K3 surfaces $\mathcal{X}, \mathcal{Y}, \mathcal{Z}$ are defined by the equations
\begin{equation}
\label{eqns_Kummer_u}
\scalemath{\MyScaleSmall}{
\begin{array}{lll}
 \mathcal{Y}_{[u_0:u_1]}:& Y^2 Z  = X \Big(X - 2(\beta-\alpha) \rho Z \Big)  \Big(X - 2(\beta+\alpha) \rho Z\Big), & \Delta_{\mathcal{Y}}=2^8 \alpha^2 \rho^6 (\beta^2-\alpha^2)^2,\\[0.5em]
 \mathcal{Z}_{[u_0:u_1]}:& V^2  =  \alpha\rho U^4 + 2\beta\rho U^2W^2 + \alpha\rho W^4, & \Delta_{\mathcal{Z}}= \Delta_{\mathcal{Y}},\\[0.5em]
 \mathcal{X}_{[u_0:u_1]}:& y^2 z  = x \big( x^2 + 2\beta \rho  x z+  \rho^2 \alpha^2z^2\big), &  \Delta_{\mathcal{X}}=2^4 \alpha^4 \rho^6 (\beta^2-\alpha^2),
\end{array}}
\end{equation} 
where we set $ \rho = (u_0^2-u_1^2)u_1$ and
\begin{equation}
\alpha= (\mu_2-\mu_3) (u_0-\mu_1 u_1), \; \beta = (2\mu_1-\mu_2-\mu_3) u_0 + (2\mu_2\mu_3-\mu_1\mu_2-\mu_1\mu_3) u_1.
\end{equation}
We have the following:
\begin{corollary}
\label{special_X}
\begin{enumerate}
\item[]
\item $\mathcal{X}$ is the double cover branched along a conic, a line tangent to the conic, and three lines coincident in a point.
\item $\mathcal{X}$ is the Kummer surface of $\operatorname{Jac}(\check{\mathcal{D}})$ with $\check{\mathcal{D}}:=\mathcal{D}(\check{\mu}_1, \check{\mu}_2, \check{\mu}_3)$.
\item $\mathcal{Y}$ is the double cover branched along six lines, three are coincident in a point.
\item $\mathcal{Z}$ is the Kummer surface of the $(1,2)$-polarized abelian surface $\check{\mathbf{B}}_{56}$.
\end{enumerate}
\end{corollary}
\begin {proof}
(1) Using Equations~(\ref{dual_moduli}), $\mathcal{X}_{[u_0:u_1]}$ in Equation~(\ref{eqns_Kummer_u}) can be brought into the affine form
\begin{equation}
 \mathcal{X}_u: \quad y^2 = x (x^2 + 2 u x + 1) \prod_{i=1}^3 (u - \check{\mu}_i).
\end{equation} 
(3) Similarly, $\mathcal{Y}_{[u_0:u_1]}$ in Equation~(\ref{eqns_Kummer_u}) can be brought into the affine form
\begin{equation}
  \mathcal{Y}_u:  \quad Y^2 =X (X-2u-1)(X-2u+1) \prod_{i=1}^3 (u - \check{\mu}_i).
\end{equation} 
(2) and (4) are a direct consequence of the Propositions~\ref{special1}, \ref{special2} and Corollary~\ref{Kum_B_12}. 
\end{proof}
\par We consider the map $h: \mathbb{P}^1_{[t_0:t_1]} \to \mathbb{P}_{(u)}^1$ given by $[t_0:t_1] \mapsto [u_0:u_1] =[t_0^2+t_1^2:2t_0t_1]$ invariant under the involution $[t_0:t_1]\mapsto [t_1:t_0]$. The polynomials $\alpha, \beta$ are then homogeneous polynomials of degree two in $\mathbb{C}(t_0,t_1)$, invariant under the involution, with no repeated roots or common factors with each other or $t_0t_1$.  Moreover, under pullback we have $\rho=t_0t_1\tilde{\rho}^2 $ with $\tilde{\rho}=\sqrt{2} (t_0^2-t_1^2)$. We obtain an elliptically fibered K3 surface $\pi_{\tilde{\mathcal{Y}}}: \tilde{\mathcal{Y}} \to \mathbb{P}^1_{[t_0:t_1]}$ in terms of the coordinates $[\tilde{X}:\tilde{Y}:\tilde{Z}]\in \mathbb{P}^2$ and $[t_0:t_1]\in \mathbb{P}^1_{[t_0:t_1]}$ from the equation 
\begin{equation}
\label{eqns_Kummer_up_Y_hat}
 \tilde{\mathcal{Y}}_{[t_0:t_1]}: \; \tilde{Y}^2 \tilde{Z}  = \tilde{X}  \Big(\tilde{X }- 2(\beta-\alpha) t_0t_1 \tilde{Z} \Big)  \Big(\tilde{X} - 2(\beta+\alpha)t_0t_1 \tilde{Z}\Big).
\end{equation} 
\par Similarly, we consider the map $g: \mathbb{P}^1_{[s_0:s_1]} \to \mathbb{P}^1_{[t_0:t_1]}$ given by $[s_0:s_1] \mapsto [t_0:t_1] =[s^2_0:s^2_1]$, invariant under the involution $[s_0:s_1]\mapsto [s_0:-s_1]$. The polynomials $\alpha, \beta$ can then be considered homogeneous polynomials of degree four in $\mathbb{C}(s_0,s_1)$, invariant under the involution, with no repeated roots or common factors with each other or $s_0s_1$.  We obtain another elliptically fibered K3 surface $\pi_{\hat{\mathcal{Y}}}: \hat{\mathcal{Y}} \to \mathbb{P}^1_{[s_0:s_1]}$ with section in terms of the coordinates $[\tilde{x}:\tilde{y}:\tilde{z}]\in \mathbb{P}^2$ for the fiber and $[s_0:s_1]\in \mathbb{P}^1_{[s_0:s_1]}$ for the base from
\begin{equation}
\label{eqns_Kummer_Z_tilde}
\scalemath{\MyScaleBig}{
\begin{array}{lll}\displaystyle
 \hat{\mathcal{Y}}_{[s_0:s_1]}:& \hat{Y}^2 \hat{Z}  = \hat{X}  (\hat{X}- 2(\beta-\alpha)  \hat{Z} )  (\hat{X} - 2(\beta+\alpha) \hat{Z}), & \Delta_{\hat{\mathcal{Y}}}=2^8 \alpha^2 (\beta^2-\alpha^2)^2.
 \end{array}}
\end{equation} 
We will now establish maps between the K3 surfaces, i.e.,
\begin{equation}
 \hat{\mathcal{Y}} \overset{G}{\dasharrow} \tilde{\mathcal{Y}}   \overset{H}{\dasharrow}  \mathcal{Y} ,
\end{equation}
that are induced by the rational base transformations $g$ and $h$, respectively.
\begin{proposition}
In the situation described above, the maps $G: \hat{\mathcal{Y}}\dasharrow \tilde{\mathcal{Y}}$ and $H: \tilde{\mathcal{Y}}\dasharrow \mathcal{Y}$ given by
 \begin{equation}
\scalemath{\MyScaleSmall}{
\begin{array}{rcl}
G:	\Big([s_0:s_1],[\hat{X}: \hat{Y}: \hat{Z}]\Big)	&  \mapsto 	
 &	\Big([t_0:t_1],[\tilde{X}:\tilde{Y}:\tilde{Z}]\Big) =\Big([s^2_0:s^2_1],[s^2_0s^2_1\hat{X}:s^3_0s^3_1\hat{Y}:\hat{Z}]\Big), \\[0.5em]
 H:	\Big([t_0:t_1],[\tilde{X}: \tilde{Y}: \tilde{Z}]\Big)	&  \mapsto 	
 &	\Big([u_0:u_1],[X:Y:Z]\Big) =\Big([t_0^2+t_1^2:2t_0t_1],[\tilde{\rho}^2\tilde{X}:\tilde{\rho}^3\tilde{Y}:\tilde{Z}]\Big),
\end{array}}
\end{equation}
are the rational maps of degree two associated with the Nikulin involutions on $\tilde{\mathcal{Y}}$ and  $\hat{\mathcal{Y}}$ given by
\begin{equation}
\begin{array}{lrcl}
  \jmath_{\hat{\mathcal{Y}}}:&  \Big([s_0:s_1],[\hat{X}: \hat{Y}: \hat{Z}]\Big) & \mapsto & \Big([s_0,-s_1],[\hat{X}: -\hat{Y}: \hat{Z}]\Big),  \\[0.5em]
  \jmath_{\tilde{\mathcal{Y}}}:&  \Big([t_0:t_1],[\tilde{X}: \tilde{Y}: \tilde{Z}]\Big) & \mapsto & \Big([t_1,t_0],[\tilde{X}: -\tilde{Y}: \tilde{Z}]\Big),
\end{array}
\end{equation}
such that $H \circ  \imath_{\tilde{\mathcal{Y}}}=H$ and $G \circ  \imath_{\hat{\mathcal{Y}}}=G$.
\end{proposition}
\begin{proof}
The proof is the same as for Proposition~\ref{diamond}.
\end{proof}
We have the following direct consequence of the Propositions~\ref{special1}, \ref{special2} and Corollaries~\ref{Kum_B_12} and \ref{even_eight_for_base_transfo}:
\begin{corollary}
\label{special_Ytilde}
\begin{enumerate}
\item[]
\item $\tilde{\mathcal{Y}}$ is the double cover of $\mathbb{P}^2$ branched along six lines tangent to a conic.
\item $\tilde{\mathcal{Y}}$ is the Kummer surface of $\operatorname{Jac}(\mathcal{D})$ with $\mathcal{D}:=\mathcal{D}(\mu_1, \mu_2, \mu_3)$.
\item $\hat{\mathcal{Y}}$  is the Kummer surface of the $(1,2)$-polarized abelian surface $\mathbf{B}_{56}$.
\end{enumerate}
\end{corollary}
\begin{proof}
(1) Using Equations~(\ref{Rosenhain_roots}), $\tilde{\mathcal{Y}}_{[t_0:t_1]}$ in Equation~(\ref{eqns_Kummer_up_Y_hat}) can be brought into the affine form
\begin{equation}
  \tilde{\mathcal{Y}}_t:  \quad \tilde{Y}^2 = t \prod_{i=1}^4 (\tilde{X} - \lambda_i t - \lambda_i^{-1}),
\end{equation} 
where we have set $\lambda_4=1$.
(2) is already proved. (3) What remains to show is that the even eight of the double cover $G: \hat{\mathcal{Y}}\dasharrow \tilde{\mathcal{Y}}$ is the even eight $\Delta_{56}$. Just as in Corollary~\ref{even_eight_for_base_transfo}, the even eight for $G$ is given by the non-central components of the reducible fibers of type $D_4$ over $t_1=0$ and $t_0=0$. It was shown in \cite{MR2804549} that these components form the even eight $\Delta_{56}$.
\end{proof}
We denote the dependence on the set of moduli $(\mu_1,\mu_2,\mu_3)$ by writing $\mathcal{X} = \mathcal{X}(\mu)$, $\mathcal{Y} = \mathcal{Y}(\mu)$, $\mathcal{Z} = \mathcal{Z}(\mu)$, etc. For example, it follows from Corollaries~\ref{special_X} and~\ref{special_Ytilde} that the following total spaces are isomorphic
\begin{equation}
 \begin{array}{c}
 \mathcal{X}(\check{\mu}) \cong \operatorname{Kum}(\operatorname{Jac}\mathcal{D}) \cong \tilde{\mathcal{Y}}(\mu), \quad
 \mathcal{X}(\mu) \cong \operatorname{Kum}(\operatorname{Jac}\check{\mathcal{D}}) \cong \tilde{\mathcal{Y}}(\check{\mu}),  \\[0.5em]
 \mathcal{Z}(\check{\mu}) \cong \operatorname{Kum}(\mathbf{B}_{56}) \cong \hat{\mathcal{Y}}(\mu).
 \end{array}
 \end{equation}
We have the following:
\begin{proposition}
In the situation described above, the elliptic fibration with section $\pi_{\hat{\mathcal{Y}}(\mu)}: \hat{\mathcal{Y}}(\mu) \to \mathbb{P}^1$ with singular fibers $12 I_2$ and $\operatorname{MW}(\pi_{\hat{\mathcal{Y}}(\mu)})=(\mathbb{Z}/2\mathbb{Z})^2+\langle 1 \rangle^{\oplus 2} +\langle 2 \rangle$ is isomorphic to the elliptic fibration with section $\check{\pi}_{\mathcal{Z}(\check{\mu})}: \mathcal{Z}(\check{\mu}) \to \mathbb{P}^1$ established in Proposition~\ref{Gamma-fibration} for the dual set of moduli.
\end{proposition}
\begin{proof}
All coefficients of $U^4, U^2W^2, W^4$ in Equation~(\ref{eqns_Kummer_u}) are polynomials in $u_0, u_1$ of degree four. We interchange the roles of base and fiber coordinates and obtain a second elliptic fibration. Since all coefficients are polynomials of degree three in the affine coordinate $u$ with $u_0=u, u_1=1$, the equation can be brought in Weierstrass normal form and this second fibration has a section. A computation shows that if we replace $\mu_1, \mu_2, \mu_3$ by $\check{\mu}_1, \check{\mu}_2, \check{\mu}_3$, this is the same Weierstrass fibration as the one given for $\tilde{\mathcal{Z}}$ in Equation~(\ref{eqns_Kummer_Z_tilde}).
\end{proof}
We also investigate the effect of interchanging base and fiber for the fibration $\pi_{\mathcal{Y}(\mu)}:\mathcal{Y}(\mu) \to \mathbb{P}^1$. The K3 surface $\mathcal{Y}_{[u_0:u_1]}$ in Equation~(\ref{eqns_Kummer_u}) can easily be brought into the affine form
\begin{equation}
\label{WEQ1}
  \mathcal{Y}(\mu):  \quad Y^2 =X (X-1) (X-u) (u-\alpha)(u-\beta)(u-\gamma),
\end{equation} 
where by abuse of notation we use the same names for the affine variables before and after the birational transformation. The parameters $\alpha, \beta, \gamma$ are rational functions of the moduli $\mu_1, \mu_2, \mu_3$. The elliptic fibration $\pi_{\mathcal{Y}(\mu)}: \mathcal{Y}(\mu) \to \mathbb{P}^1$ is simply the projection onto $[u:1]\in\mathbb{P}^1$. The fibration has a section because Equation~(\ref{WEQ1}) can be brought into Weierstrass normal form. A \emph{second} elliptic fibration with section is induced by projecting onto $[X:1]\in\mathbb{P}^1$ which we denote by $\check{\pi}_{\mathcal{Y}(\mu)}: \mathcal{Y}(\mu) \to \mathbb{P}^1$. We have the following:
\begin{proposition}
In the situation described above, the elliptic fibration with section $\pi_{\mathcal{Y}(\mu)}: \mathcal{Y}(\mu) \to \mathbb{P}^1$ with singular fibers $3 I_0^* + 3 I_2$ and $\operatorname{MW}(\pi_{\mathcal{Y}(\mu)})=(\mathbb{Z}/2\mathbb{Z})^2$ is isomorphic to the second elliptic fibration with section $\check{\pi}_{\mathcal{Y}(\check{\mu})}: \mathcal{Y}(\check{\mu}) \to \mathbb{P}^1$ established above for the dual set of moduli. In particular, we have $\mathcal{Y}(\mu)\cong\mathcal{Y}(\check{\mu})$.
\end{proposition}
\begin{proof}
Equation~(\ref{WEQ1}) can be brought into the form
\begin{equation}
\label{WEQ2}
  \mathcal{Y}(\mu): \quad Y^2 =u (u-1) (u-X) (X-\check{\alpha})(u-\check{\beta})(u-\check{\gamma}).
\end{equation}
Both set of parameters are functions of the moduli $\mu_1, \mu_2, \mu_3$, i.e., $\alpha=\alpha(\mu_1, \mu_2, \mu_3)$, $\check{\alpha}=\check{\alpha}(\mu_1, \mu_2, \mu_3)$, etc.
Using Equation~(\ref{dual_moduli}), one checks that the transformation $\alpha(\check{\mu}_1, \check{\mu}_2, \check{\mu}_3) = \check{\alpha}(\mu_1, \mu_2, \mu_3)$, and the same hold for $\beta$ and $\check{\beta}$ and $\gamma$ and $\check{\gamma}$.
\end{proof}
\noindent
The situation is depicted in the following diagram where the double arrows denote isomorphisms obtained by interchanging base and fiber coordinates:
\begin{equation}
\scalemath{\MyScaleTiny}{
\xymatrix{
& \hat{\mathcal{Y}}(\mu)= \operatorname{Kum}(\mathbf{B}_{56}) \ar[dl] \ar@{<=>}[ddr] &  \hat{\mathcal{Y}}(\check{\mu})= \operatorname{Kum}(\check{\mathbf{B}}_{56})  \ar[dr]\ar@{<=>}[ddl]|\hole\\
 \tilde{\mathcal{Y}}(\mu)= \operatorname{Kum}(\operatorname{Jac} \mathcal{D}) \ar[d] &&& \tilde{\mathcal{Y}}(\check{\mu})= \operatorname{Kum}(\operatorname{Jac} \check{\mathcal{D}}) \ar[d]\\
 \mathcal{Y}(\mu) \ar[dr] \ar@/_2pc/@{<=>}[rrr] &   \mathcal{Z}(\mu)=\operatorname{Kum}(\check{\mathbf{B}}_{56}) \ar[d]|\hole &  \mathcal{Z}(\check{\mu})  =\operatorname{Kum}(\mathbf{B}_{56}) \ar[d]|\hole & \mathcal{Y}(\check{\mu}) \ar[dl] \\ 
& \mathcal{X}(\mu) = \operatorname{Kum}(\operatorname{Jac} \check{\mathcal{D}}) &  \mathcal{X}(\check{\mu}) = \operatorname{Kum}(\operatorname{Jac} {\mathcal{D}}) & 
}}
\end{equation}
\section{Application to F-theory/heterotic string duality}\label{F-theory}
Compactifications of the type-IIB string theory in which the complex coupling varies over a base are generically referred to as F-theory. F-theory is a powerful tool for analyzing the non-perturbative aspects of heterotic string compactifications \cite{MR1409284,MR1412112}. One of the simplest F-theory construction corresponds to K3 surfaces that are elliptically fibered over $\mathbb{P}^1$, in physics equivalent to type-IIB string theory compactified on $\mathbb{P}^1$ and hence eight-dimensional in the presence of 7-branes and Wilson lines \cite{MR3366121}. In this way, an elliptically fibered K3 surface with section and elliptic fibers $F_{\pmb{\tau} }=\mathbb{C}/(\mathbb{Z}\oplus \mathbb{Z}\pmb{\tau} )$ defines an F-theory vacuum in eight dimensions where the complex-valued scalar field $\pmb{\tau}$ of the type-IIB string theory, the so-called axio-dilaton field, is now allowed to be multi-valued and undergo monodromy when encircling defects of co-dimension one. The Kodaira-table of singular fibers \cite{MR0184257} gives a precise dictionary between the characteristics of the elliptic fibration and the content of the 7-branes present in the physical theory and the local monodromy of $\pmb{\tau}$. It is well-known that the moduli space of these F-theory models is isomorphic to the moduli space of the heterotic string compactified on an elliptic curve together with a principal $G$-bundle where $G$ is the gauge group of the heterotic string, i.e., $G=E_8 \times E_8 \rtimes \mathbb{Z}_2$ or $\mathrm{Spin}(32)/\mathbb{Z}_2$ or a subgroup of these. In fact, it is well-known that the moduli spaces for these two types physical theories are given by the Narain space which is the quotient of the symmetric space for $\operatorname{O}(2,18)$ by a particular arithmetic group \cite{MR834338}. This is the basic form of the F-theory/heterotic string duality.
\par The duality map can be put on a geometric footing and made more explicit if we restrict to a low-dimensional subspace of the full moduli space which in general will still include quantum compactifications. For high Picard number of the elliptically fibered K3 surfaces or, equivalently, the presence of at most one non-trivial Wilson line, the F-theory/heterotic string duality is a Hodge-theoretic duality map -- a correspondence that relates K3 surfaces and abelian surfaces. In fact, the analytic isomorphism between the moduli spaces of certain families of lattice polarized K3 surfaces and flat bundles over elliptic curves with gauge group $G$ can be explicitly computed in terms of period integrals. Using this duality, non-geometric compactifications of the heterotic string on a two-torus with no large radius interpretation were constructed by using the F-theory dual, together with a close connection between Siegel modular forms of genus two and the equations of certain K3 surfaces \cite{MR3366121}.  
\subsection{The CHL string in eight dimensions}
\label{sec:CHL}
In eight dimensions one can also consider the so-called \emph{CHL string}. Here, we give a brief overview of the CHL string and an outline of a class of dualities linked to CHL string. A detailed discussion of the CHL string and the dual eight-dimensional F-theory models can be found in \cite{MR1615617, MR1621170, MR1797021}. The CHL string is obtained from the $E_8 \times E_8$ heterotic string on a torus $T^2$ as a $\mathbb{Z}/2\mathbb{Z}$ quotient. The quotient is obtained from an involution, called the \emph{CHL involution}. The CHL involution acts by a half-period shift on the elliptic curve obtained by combining $T^2$ with the complex structure parameter, permutes the two $E_8$'s of the gauge bundle, and acts trivially on the complex K\"ahler modulus.  Keeping the notation from Section~\ref{setup}, we can start with the following data: ($a'$) an elliptic curve $\mathcal{E}$, ($b'$) a $\mathbb{Z}/2\mathbb{Z}$ subgroup  $\{\sigma, \tau\}$ of $\mathcal{E}$ inducing the involution $\imath_{\mathcal{E}}$ in Equation~(\ref{involution_fiber}), and ($c'$) a principal $E_8$ bundle $\mathbf{W} \to \mathcal{E}$. The actual data of the CHL string is then given by \cite{MR1797021}: ($a$) the elliptic curve $\hat{\mathcal{E}} = \mathcal{E}/\{ \sigma, \tau\}$ in Equation~(\ref{fiber_isog}), ($b$) the point $T$ of order two on $\hat{\mathcal{E}}$ inducing the CHL involution $\imath_{\hat{\mathcal{E}}}$ in Equation~(\ref{dual_isog_involution}), and ($c$) the twisted principal $E_8 \times E_8$ bundle $\mathbf{V} \to \hat{\mathcal{E}}$ obtained from the push-forward map $\hat{\varphi}_*\mathbf{W}$. Here, `twisted' means that after pullback the bundle $\hat{\varphi}^*\mathbf{W}$ becomes an ordinary $E_8\times E_8$ bundle on $\mathcal{E}$ isomorphic to $\mathbf{W} \oplus \imath_{\mathcal{E}}^*\mathbf{W} \to \mathcal{E}$.
\par The situation for the CHL string is different from the standard F-theory/heterotic string correspondence: Witten analyzed the F-theory compactifications dual to the CHL string in the case of an isotrivial elliptic fibration on the F-theory background \cite{MR1615617}.  His picture was then extended in \cite{MR1797021} to the interior of the Narain moduli space of toroidal compactifications where the elliptic fibrations of the F-theory models are no longer isotrivial, but the essential features of Witten's description are still valid. Generally speaking, the CHL string in eight dimensions is dual to an F-theory with non-zero flux of an antisymmetric two-form field, or $B$-field, along the base curve $\mathbb{P}^1$. The value of this flux is quantized and fixed to half the K\"ahler class on $\mathbb{P}^1$. According to Witten \cite{MR1615617}, the non-trivial $\mathbb{Z}/2\mathbb{Z}$ flux constitutes the obstruction to the existence of a vector structure in the physical theory. The presence of this flux freezes eight of the moduli in the physical moduli space, leaving a ten dimensional moduli space. 
\par On this ten dimensional moduli space, the single-valued background of an antisymmetric two-form in the physical theory is not compatible with a generic monodromy of the half-periods of a generic elliptic fiber in $\operatorname{SL}(2,\mathbb{Z})$ around the defects. Rather it must be contained in a subgroup of $\operatorname{SL}(2,\mathbb{Z})$ that keeps the flux of the $B$-field invariant, the biggest possible being the congruence subgroup $\Gamma_0(2)$. The mechanism of specializing the usual F-theory/heterotic string duality then works as follows: one starts with an elliptically fibered K3 surface $\mathbf{X}$ with section and a monodromy group contained in $\Gamma_0(2)$ that keeps one of the three half-periods of the elliptic fibers invariant. The corresponding Weierstrass model has eight fibers of Kodaira type $I_2$ and two sections, the zero-section and a two-torsion section, arranged in such a way that at each reducible fiber of type $A_1$ the two sections pass trough two different components of the fiber. This fibration is called the \emph{$\Gamma_0(2)$ elliptic fibration} in physics \cite{MR1797021}, since the monodromy group is $\Gamma_0(2)$. It turns out that the moduli space of F-theory compactifications corresponding to the CHL string is a finite cover of the moduli space of these Weierstrass models. Concretely, in order to recover the dual CHL compactification from the Weierstrass model one needs to choose four of the eight points on the base where the fibers of type $I_2$ are located. The double cover of $\mathbb{P}^1$ branched at these four marked points defines the genus-one curve $\hat{\mathcal{C}}$ on which the CHL string is compactified. In conclusion, the data needed to describe an F-theory compactifications dual to the CHL string is a Weierstrass model defining a $\Gamma_0(2)$ elliptic fibration with section and two-torsion section together with a choice of four out of eight fibers of $I_2$.
\par The identification of points in the physical moduli space corresponding to a specific non-abelian gauge group of the CHL string simplifies considerably, if the elliptically fibered K3 surface $\mathbf{X}$ admits a second elliptic fibration with two fibers of Kodaira type $I_0^*$. The existence of this second elliptic fibration is due to the fact that we expect the K3 surface $\mathbf{X}$ to arise as quotient of another K3 surface $\mathbf{Y}$ by a Nikulin involution. Physically, the existence of the second K3 surface implies that the F-theory/heterotic string duality can be lifted to a duality between the CHL string and the seven-dimensional compactification of M-theory on the quotient of $\mathbf{Y} \times S^1$ by the product involution acting as Nikulin involution on $\mathbf{Y}$ and as half-period shift on $S^1$. We assume that $\mathbf{Y}$ admits an elliptic fibration with section such that $\mathbf{X}$ inherits an elliptic fibration with two fibers of Kodaira type $I_0^*$, called the \emph{inherited elliptic fibration} in physics~\cite{MR1797021}. However, the correct F-theory compactification does not admit a section but only a bi-section due to the precise nature of the duality with the CHL string.  That is, the correct F-theory compactification is that of a K3 surface $\mathbf{X}$ equipped with a genus-one fibration with two fibers of Kodaira type $I_0^*$ \emph{not} admitting a section. The connection to the $\Gamma_0(2)$ elliptic fibration is seen as follows: in the $\Gamma_0(2)$ elliptic fibration there are two sections, the zero section and a two-torsion section, as well as eight $A_1$ singularities. Of those, four singularities are located on the two-torsion section, four more singularities are located on the zero-section. The two sections become the central components of the two $I_0^*$ fibers in the inherited fibration. By the nature of the duality, the gauge groups of the CHL string can then be easily read off the singular fibers in the inherited elliptic fibration \cite{MR1797021}. In fact, Witten argued that the moduli space of F-theory compactifications dual to the CHL string is naturally isomorphic to the moduli space of K3 surfaces obtained as genus-one fibrations with two fibers of type $I_0^*$ and a bisection \cite{MR1615617}.
\par There is also third point of view regarding the F-theory backgrounds dual to the CHL string: since the CHL string is essentially a heterotic string theory with involution, one expects the duality to equip the K3 surface $\mathbf{Y}$ with an involution as well. However, the existence of the two elliptic fibration structures on $\mathbf{X}$ implies that there is also an involution on $\mathbf{Y}$ acting fixed-point-free. The quotient of the K3 surface $\mathbf{Y}$ by this fixed-point-free involution is an \emph{Enriques surface} that admits a bi-section of arithmetic genus one and two fibers of multiplicity two. There is a well understood moduli space of K3 surface with this property, namely the moduli space of K3 surfaces which are universal covers of Enriques surfaces \cite{MR3586505}.
\par We employ the same strategy as the one applied in~\cite{MR3366121} for the standard F-theory/heterotic string correspondence. To determine explicitly the moduli for a family of F-theory compactifications with non-zero flux in eight dimensions and compactifications of the CHL string, we focus on a natural subspace of the full moduli space. As mentioned, it was proven in~\cite{MR2274533} that the most generic K3 surfaces in Section~\ref{sec:generic} allowed by Theorem~\ref{thm} have the singular fibers $8 I_2 + 8I_1$, admit a lattice polarization by the rank-ten lattice $H \oplus N$ -- where $N$ is the Nikulin lattice -- and have a ten-dimensional moduli space. We will then investigate a natural sub-family of K3 surfaces in Picard rank $14$ in more detail.  To do so, we restrict ourselves to the case when there are two commuting involutions on $\mathbf{Y}$, one preserving the fibers and one acting as simple inversion on the base. The additional symmetry restricts us to a six dimensional moduli space of a family of lattice polarized K3 surfaces of Picard rank $14$ that will be determined in Theorem~\ref{prop:CHL} below. For this family, the F-theory/heterotic string duality map and moduli are then determined explicitly in Theorem~\ref{prop:duality} below. 
\subsection{F-theory moduli in Picard rank 14}
We specialize Theorem~\ref{thm} to the case $k=2$ and assume that $a= t_0t_1\alpha$, $b=t_0t_1\beta$, $c=t_0t_1\gamma$, where $\alpha, \beta, \gamma$ are homogeneous polynomials of degree two in $\mathbb{C}(t_0,t_1)$ with no repeated roots or common factors with each other or $t_0t_1$.  We obtain projective models for K3 surfaces $\mathcal{X}, \mathcal{Y}, \mathcal{Z}$ with coordinates $[x:y:z], [X:Y:Z]\in \mathbb{P}^2$, $[U:V:W]\in \mathbb{P}(1,2,1)$ for the fiber and $[t_0:t_1]\in \mathbb{P}^1$ for the base from the equations
\begin{equation}
\label{eqns_Kummer2}
\scalemath{\MyScaleMedium}{
\begin{array}{lll}
 \mathcal{Y}_{[t_0:t_1]}:& Y^2 Z  = X \Big(X^2 - 4 t_0t_1 \beta XZ + 4t_0^2t_1^2(\beta^2-\alpha\gamma)Z^2\Big), & \Delta_{\mathcal{Y}}=2^8 \alpha \gamma (\beta^2-\alpha\gamma)^2 t^6_0 t^6_1,\\[0.5em]
  \mathcal{Z}_{[t_0:t_1]}:& V^2  =  t_0 t_1 \alpha  U^4 +2  t_0 t_1 \beta  U^2W^2 +  t_0 t_1 \gamma  W^4, & \Delta_{\mathcal{Z}}= \Delta_{\mathcal{Y}},\\[0.5em]
 \mathcal{X}_{[t_0:t_1]}:& y^2 z  = x \big( x^2 + 2 t_0 t_1 \beta  x z+  t^2_0 t_1 ^2 \alpha\gamma z^2\big), &  \Delta_{\mathcal{X}}=4\alpha^2\gamma^2(\beta^2-\alpha\gamma) t^6_0 t^6_1.
\end{array}}
\end{equation} 
Condition~(\ref{condition}) is not satisfied, and the genus-one fibration $\pi_{\mathcal{Z}}$ has no sections. The elliptic fibrations with section $\pi_{\mathcal{X}}$ and $\pi_{\mathcal{Y}}$ have singular fibers $2 I_0^* + 4 I_2 + 4 I_1$ and $\operatorname{MW}(\pi_{\mathcal{Y}})=\operatorname{MW}(\pi_{\mathcal{X}})=\mathbb{Z}/2\mathbb{Z}$. Accordingly, we have $\rho_{\mathcal{X}} = \rho_{\mathcal{Y}}= \rho_{\mathcal{Z}} = 14$.  The K3 surface $\mathcal{X}$ admits the zero-section $\sigma: [x:y:z]=[0:1:0]$ and the two-torsion section $\tau: [x:y:z]=[0:0:1]$.
\par As shown in Figure~(\ref{pic:VGS}), the K3 surfaces $\mathcal{X}$ and $\mathcal{Y}$ are related to each other by fiberwise two-isogeny. The ramification locus of the double cover $\Phi: \mathcal{Y} \dasharrow \mathcal{X}$ is the even eight that consists of the  four non-neutral components of the reducible fibers of type $A_1$ and the four non-central components of the reducible fibers of type $D_4$ over $t_0t_1=0$ not meeting the sections $\sigma$ or $\tau$. A rational map $\Psi: \mathcal{Z}\dasharrow \mathcal{X}$ of degree two is given by
\begin{equation}
\label{ppZ}
\begin{split}
  [x:y:z] & =\Psi([U:V:W])= [t_0t_1\alpha U^2W: t_0t_1\alpha U V : W^3 ].
\end{split}
\end{equation} 
We have the following:
\begin{corollary}
\label{prop_ppZ}
In the situation described above, the map $\Psi: \mathcal{Z} \dasharrow \mathcal{X}$ is the double cover branched along the even eight on $\mathcal{X}$ given by the following components: the two components of the reducible fibers of type $A_1$ over $\gamma=0$ not meeting the section $\sigma$, the two components of the reducible fibers of type $A_1$ over $\alpha=0$ meeting the section $\sigma$, the four non-central components of the reducible fibers of type $D_4$ over $t_0t_1=0$ meeting sections $\sigma$ or $\tau$. 
\end{corollary}
\begin{proof}
The proof is obtained by combining the proofs of Propositions~\ref{prop_EEP_I4} and \ref{prop_EEP_D4}. 
\end{proof}
\noindent Since the even eight, i.e., the branching locus of $\Psi$ on $\mathcal{X}$, is intersected by the zero-section and two-torsion section, none of these sections lift to sections on $\mathcal{Z}$ which generically has no sections. The genus-one fibration $\pi_{\mathcal{Z}} : \mathcal{Z} \to \mathbb{P}^1$ has the same $J$-function and the same monodromy as the elliptically fibered K3 surface $\pi_{\mathcal{Y}} : \mathcal{Y} \to \mathbb{P}^1$ with section, but is not in Weierstrass form. Conversely, the elliptic fibration $\pi_{\mathcal{Y}}$ with section is the relative Jacobian fibration of $\pi_{\mathcal{Z}}$ by Proposition~\ref{prop_K3}.

\par We consider the double cover $f: \mathbb{P}^1 \to \mathbb{P}^1$ given by $[s_0:s_1] \mapsto [t_0:t_1] =[s^2_0:s^2_1]$ invariant under the involution $[s_0:s_1]\mapsto [s_0:-s_1]$. The polynomials $\alpha, \beta$ are then homogeneous polynomials of degree four (resp.~\!2) in $\mathbb{C}(s_0,s_1)$, invariant under the involution, with no repeated roots or common factors with each other or $s_0s_1$.  K3 surfaces $\tilde{\mathcal{X}}, \tilde{\mathcal{Y}}, \tilde{\mathcal{Z}}$ are obtained using coordinates $[\tilde{x}:\tilde{y}:\tilde{z}]\in \mathbb{P}^2$, $[\tilde{X}:\tilde{Y}:\tilde{Z}]\in \mathbb{P}^2$, $[U:V:W]\in \mathbb{P}(1,2,1)$ for the fibers and $[s_0:s_1]\in \mathbb{P}^1$ for the base from the equations
\begin{equation}
\label{eqns_Kummer_up2}
\scalemath{\MyScaleBig}{
\begin{array}{lll}
 \tilde{\mathcal{Y}}_{[s_0:s_1]}:& \tilde{Y}^2 \tilde{Z}  = \tilde{X} \Big(\tilde{X}^2 - 4 \beta \tilde{X} \tilde{Z} + 4(\beta^2-\alpha\gamma) \tilde{Z} ^2\Big), & \Delta_{\tilde{\mathcal{Y}}}=2^8 \alpha \gamma (\beta^2-\alpha\gamma)^2,\\[0.5em]
 \tilde{\mathcal{Z}}_{[s_0:s_1]}:& \tilde{V}^2  =  \alpha  \tilde{U}^4 +2  \beta  \tilde{U}^2\tilde{W}^2 +  \gamma  \tilde{W}^4, & \Delta_{\tilde{\mathcal{Z}}}= \Delta_{\tilde{\mathcal{Y}}},\\[0.25em]
 \tilde{\mathcal{X}}_{[s_0:s_1]}: & \tilde{y}^2 \tilde{z}  = \tilde{x}  \Big(\tilde{x}^2 + 2 \beta  \tilde{x} \tilde{z}+  \alpha\gamma \tilde{z}^2\Big), &  \Delta_{\tilde{\mathcal{X}}}=4\alpha^2\gamma^2(\beta^2-\alpha\gamma) .
\end{array}}
\end{equation} 
Condition~(\ref{condition}) is not satisfied, and the genus-one fibration $\pi_{\tilde{\mathcal{Z}}}$ has no sections. The elliptic fibrations with section $\pi_{\tilde{\mathcal{X}}}$ and $\pi_{\tilde{\mathcal{Y}}}$ have singular fibers $8 I_2 + 8 I_1$ and $\operatorname{MW}(\pi_{\mathcal{Y}})_{\mathrm{tor}}=\operatorname{MW}(\pi_{\mathcal{X}})_{\mathrm{tor}}=\mathbb{Z}/2\mathbb{Z}$. Following Proposition~\ref{diamond} and Corollary~\ref{even_eight_for_base_transfo}, there is a degree-two rational map $F: \tilde{\mathcal{Y}}\dasharrow \mathcal{Y}$ given by
\begin{equation*}
 F:	\Big([s_0:s_1],[\tilde{X}: \tilde{Y}: \tilde{Z}]\Big)	  \mapsto 	
 	\Big([t_0:t_1],[X:Y:Z]\Big) =\Big([s^2_0:s^2_1],[s_0^2s_1^2\tilde{X}:s_0^3s_1^3\tilde{Y}:\tilde{Z}]\Big), 
 \end{equation*}
associated with the Nikulin involution on $\tilde{\mathcal{Y}}$ given by
\begin{equation}
\label{Nikulin_base}
  \jmath_{\tilde{\mathcal{Y}}}: \;  \Big([s_0:s_1],[\tilde{X}: \tilde{Y}: \tilde{Z}]\Big) \; \mapsto \; \Big([s_0:-s_1],[\tilde{X}: -\tilde{Y}: \tilde{Z}]\Big).
\end{equation}
The branching locus of the map $F$ is the even eight on $\mathcal{Y}$ given by the non-central components of the reducible fibers of type $D_4$ over $t_1=0$ and $t_0=0$. 
\par After choosing one of the polynomials from the pair $\{\alpha, \gamma\}$, say $\alpha$, corresponding to a choice of four out of the eight fibers of type $I_2$, and rescaling 
\begin{equation}
\label{birational_change}
\Big( \tilde{x}, \tilde{y}\Big) \mapsto \Big( \alpha \tilde{x}, \alpha \tilde{y}\Big),
\end{equation}
Equation~(\ref{eqns_Kummer_up2}) becomes
\begin{equation}
\label{eqns_Kummer_up2b}
 \tilde{y}^2 \tilde{z}  = \tilde{x}  \Big(\alpha \tilde{x}^2 + 2 \beta  \tilde{x} \tilde{z}+  \gamma \tilde{z}^2\Big).
\end{equation}
In Equation~(\ref{eqns_Kummer_up2b}) the projection to $[\tilde{x}:\tilde{z}] \in \mathbb{P}^1$ induces a second, genus-one fibration $\check{\pi}_{\tilde{\mathcal{X}}}$ because the right hand side of Equation~(\ref{eqns_Kummer_up2b}) is homogeneous of degree four in $\mathbb{C}(s_0,s_1)$.  Similarly, the K3 surface $\tilde{\mathcal{Z}}$ in Equation~(\ref{eqns_Kummer_up2}) is equipped with a second genus-one fibrations $\check{\pi}_{\tilde{\mathcal{Z}}}: \tilde{\mathcal{Z}} \to \mathbb{P}^1$ induced by the projection to  $[\tilde{U}:\tilde{W}] \in \mathbb{P}^1$. We encountered $\check{\pi}_{\tilde{\mathcal{Z}}}$ already in higher Picard number in Proposition~\ref{Gamma-fibration}.
\par We set
\begin{equation}
\label{coeffs2}
\scalemath{\MyScaleMedium}{
\begin{array}{c}\displaystyle
 \alpha = \alpha_2 t_0^2 +2 \alpha_1 t_0t_1 + \alpha_0 t_1^2, \quad  \beta = \beta_2 t_0^2 + 2\beta_1 t_0 t_1 + \beta_0 t_1^2, \quad
 \gamma = \gamma_2 t_0^2 + 2\gamma_1 t_0t_1 + \gamma_0 t_1^2,
 \end{array}}
\end{equation}
and denote the dependence on the set of moduli $(\alpha_1, \alpha_2, \alpha_3, \beta_1 , \dots, \gamma_3)$ by writing $\mathcal{X} = \mathcal{X}(\alpha, \beta, \gamma)$, $\mathcal{Y} = \mathcal{Y}(\alpha, \beta, \gamma)$, $\mathcal{Z} = \mathcal{Z}(\alpha, \beta, \gamma)$, etc. We also define a `dual' set of moduli, denoted by $(\check{\alpha}_1, \check{\alpha}_2, \check{\alpha}_3, \check{\beta}_1 , \dots, \check{\gamma}_3)$, related to $(\alpha, \beta, \gamma)$ by
\begin{equation}
\label{dual_6L}
\begin{array}{c}
 \check{\alpha}_2 = \alpha_2, \; \check{\alpha}_1= \beta_2, \; \check{\alpha}_0 = \gamma_2, \qquad \check{\beta}_2 = \alpha_1, \; \check{\beta}_1=\beta_1, \; \check{\beta}_0 =\gamma_1,\\[0.5em]
 \check{\gamma}_2 = \alpha_0, \; \check{\gamma}_1= \beta_0, \; \check{\gamma}_0 =\gamma_0.
 \end{array}
\end{equation}
It is obvious that the relation between the moduli $(\alpha, \beta,\gamma)$ and $(\check{\alpha}, \check{\beta},\check{\gamma})$ is bijective and symmetric. We have the following:
\begin{lemma}
\label{lem:EquivFibs}
\begin{enumerate}
\item []
\item The fibration $\pi_{\tilde{\mathcal{Z}}(\alpha, \beta, \gamma)}$ in Equation~(\ref{eqns_Kummer2}) coincides with the fibration $\check{\pi}_{\tilde{\mathcal{Z}}(\check{\alpha}, \check{\beta}, \check{\gamma})}$. In particular, the total spaces of the elliptic fibrations are isomorphic, i.e., $\tilde{\mathcal{Z}}(\alpha, \beta, \gamma) \cong \tilde{\mathcal{Z}}(\check{\alpha}, \check{\beta}, \check{\gamma})$.
\item The fibration $\pi_{\mathcal{Z}(\check{\alpha}, \check{\beta}, \check{\gamma})}$ in Equation~(\ref{eqns_Kummer2}) coincides with the fibration $\check{\pi}_{\tilde{\mathcal{X}}(\alpha, \beta, \gamma)}$. In particular, we have $\mathcal{Z}(\check{\alpha}, \check{\beta}, \check{\gamma}) \cong \tilde{\mathcal{X}}(\alpha, \beta, \gamma)$.
\end{enumerate}
\end{lemma}
\begin{proof}
We only proof (2): one checks that for $\tilde{z}=1$ replacing $\tilde{x} \mapsto t$, $\tilde{y} \mapsto V$, $s \mapsto U$ in Equation~(\ref{eqns_Kummer_up2b}) one obtains $\mathcal{Z}$ in Equation~(\ref{eqns_Kummer2}) with $W=1$ if the parameters of $\mathcal{Z}$ are replaced by $(\check{\alpha}, \check{\beta}, \check{\gamma})$ according to Equation~(\ref{dual_6L}). 
\end{proof}
\begin{theorem}
\label{prop:CHL}
The K3 surface $\tilde{\mathcal{X}}$ admits a lattice polarization by the lattice $H \oplus N \oplus D_4^*(-2)$ -- where $N$ is the rank-eight Nikulin lattice and $D_4^*(-2)$ is the dual of the root lattice of type $D_4$ whose intersection form has been rescaled by $(-2)$ -- and two elliptic fibrations: the first fibration $\pi_{\tilde{\mathcal{X}}}$ has a section and a two-torsion section with singular fibers $8 I_2 + 8 I_1$ and monodromy group $\Gamma_0(2)$; the second fibration is a genus-one fibration $\check{\pi}_{\tilde{\mathcal{X}}}$ with singular fibers $2 I_0^* + 4 I_2 + 4 I_1$ and does not admit a section.  
\end{theorem}
\begin{proof}
The Kodaira type of the singular fibers can be read off from the Weierstrass model in Equation~(\ref{eqns_Kummer_up2b}). Following \cite{MR1813537}, an elliptic K3 surface with root lattice $8 A_1$ can only have trivial torsion or $\mathbb{Z}/2\mathbb{Z}$. From Equation~(\ref{eqns_Kummer_up2b}) we observe that there is in fact one obvious two-torsion section. The full structure of the Mordell-Weil group $\operatorname{MW}(\pi_{\tilde{\mathcal{X}}}) =\mathbb{Z}/2\mathbb{Z} \oplus D_4^*(-2)$ then follows from Lemma~\ref{fullMW}. Since the dual of the narrow Mordell-Weil lattice does not contain any roots, the statement about the lattice polarization follows from the statement that the generic family with eight fibers of Kodaira type $I_2$ and eight fibers of type $I_1$ admits the lattice polarization $H \oplus N$ as proven in~\cite{MR2274533}. The monodromy matrices around all fibers of Kodaira type $I_2$ over $\Delta_{\tilde{\mathcal{X}}}=0$ are conjugate to the matrix $\left(\begin{smallmatrix} 1&1\\ 0&1 \end{smallmatrix}\right) \in \Gamma_0(2)$. The properties of the second fibration are the properties of $\pi_{\mathcal{Z}(\check{\alpha}, \check{\beta}, \check{\gamma})}$ when Lemma~\ref{lem:EquivFibs} is applied.
\end{proof}
Theorem~\ref{prop:CHL} proves that the K3 surfaces $\mathbf{X}:=\tilde{\mathcal{X}}(\alpha, \beta, \gamma)$ admit two descriptions: (1) the description as the $\Gamma_0(2)$ elliptic fibration $\pi_{\mathbf{X}} : \mathbf{X} \to \mathbb{P}^1$ with section and two-torsion, with singular fibers $8 I_2 + 8 I_1$, and Mordell-Weil group $\mathbb{Z}/2\mathbb{Z} \oplus D_4^*(-2)$ in Equation~(\ref{eqns_Kummer_up2}), such that the surface is a double cover of $\mathbb{F}_4$; (2) the description as genus-one fibration which does not admit a section but a bi-section, with singular fibers $2 I_0^* + 4 I_2 + 4 I_1$ -- corresponding to the reducible fibers $D_4^{\oplus 2} \oplus A_1^{\oplus 4}$, such that the surface is a double cover of $\mathbb{F}_0$. The birational map relating the two fibrations interchanges the coordinates of the base and the fiber and can only be applied after choosing four out of the eight base points of the fibers of type $I_2$ and changing coordinates according to Equation~(\ref{birational_change}). The relative Jacobian fibration of this genus-one fibration with two singular fibers of type $I_0^*$ is inherited from the elliptic fibration with section on the K3 surface $\mathbf{Y}:=\tilde{\mathcal{Y}}(\check{\alpha}, \check{\beta}, \check{\gamma})$ as quotient by the Nikulin involution in Equation~(\ref{Nikulin_base}). We have proved the following:
\begin{corollary}
The triple $(\mathbf{X}, \pi_{\mathbf{X}}, \alpha)$ form eight-dimensional F-theory compactifications with Picard number $14$ dual to the CHL string compactified on the genus-one curve $\hat{\mathcal{C}}_{\alpha}$ given by
\begin{equation}
\label{curve1}
  \hat{\mathcal{C}}_{\alpha}: \quad V^2 = \alpha(U^2,1) = \alpha_2 U^4 +2 \alpha_1 U^2 + \alpha_0.
\end{equation}
\end{corollary}
\begin{remark}
Choosing the polynomial $\gamma$ instead of $\alpha$ from the pair $\{\alpha, \gamma\}$ corresponds to choosing the four complimentary points out of eight fibers of type $I_2$, hence results in the exchange $\alpha_i \leftrightarrow \gamma_i$ for $1\le i \le 3$. The heterotic string is then compactified on the genus-one curve $\hat{\mathcal{C}}_{\gamma}$ given by
\begin{equation}
\label{curve2}
  \hat{\mathcal{C}}_{\gamma}: \quad V^2 = \gamma(U^2,1) = \gamma_2 U^4 +2 \gamma_1 U^2 + \gamma_0.
\end{equation}
\end{remark}
In the construction of the Weierstrass models in Equation~(\ref{eqns_Kummer2}) we used fractional linear transformations to move the singular fibers of type $I_0^*$ to $t_0=0$ and $t_1=0$, respectively. The remaining freedom in changing coordinates on the base is a scaling. Freedom of rescaling the fiber or base coordinates in Equation~(\ref{eqns_Kummer2}) and the ambiguity of factoring $\alpha\gamma$ determines an action of $(\lambda,\mu,\nu) \in (\mathbb{C}^{*})^3$ on the moduli given by
\begin{equation}
\label{scalingWEq}
\scalemath{\MyScaleSmall}{
\begin{array}{c}
 \Big( \alpha_2, \alpha_1, \alpha_0;  \beta_2, \beta_1, \beta_0; \gamma_2, \gamma_1, \gamma_0 \Big)
 \mapsto
  \left( \lambda\mu\nu \alpha_2, \lambda\mu \alpha_1, \frac{\lambda\mu}{\nu} \alpha_0;  \lambda\nu \beta_2, \lambda \beta_1, \frac{\lambda}{\nu} \beta_0;  \frac{\lambda\nu}{\mu} \gamma_2, \frac{\lambda}{\mu}\gamma_1, \frac{\lambda}{\mu\nu}\gamma_0 \right),
\end{array}}
\end{equation}
and analogously on the dual moduli by
\begin{equation}
\label{scalingWEq_dual}
\scalemath{\MyScaleSmall}{
\begin{array}{c}
 \Big( \check{\alpha}_2, \check{\alpha}_1, \check{\alpha}_0;  \check{\beta}_2, \check{\beta}_1, \check{\beta}_0; \check{\gamma}_2, \check{\gamma}_1, \check{\gamma}_0 \Big)
 \mapsto
  \left( \lambda\mu\nu \check{\alpha}_2, \lambda\nu \check{\alpha}_1, \frac{\lambda\nu}{\mu} \check{\alpha}_0;  \lambda\mu \check{\beta}_2, \lambda \check{\beta}_1, \frac{\lambda}{\mu} \check{\beta}_0;  \frac{\lambda\mu}{\nu} \check{\gamma}_2, \frac{\lambda}{\nu}\check{\gamma}_1, \frac{\lambda}{\mu\nu}\check{\gamma}_0 \right).
\end{array}}
\end{equation}
The following is obvious:
\begin{lemma}
The scaling in Equation~(\ref{scalingWEq}) acts by holomorphic isomorphism on the genus-one curves $\hat{\mathcal{C}}_{\alpha}$ and $\hat{\mathcal{C}}_{\gamma}$.
\end{lemma}
As $\alpha$ and $\gamma$ cannot have a common factor with $t_0t_1$, the scaling in Equation~(\ref{scalingWEq}) can be used to fix $\alpha_2=\gamma_0=1$. The residual one-parameter symmetry remaining combines a rescaling of the base with a rescaling of the fiber resulting in
\begin{equation}
\label{scalingWEq_2}
\scalemath{\MyScaleSmall}{
\begin{array}{c}
 ( \alpha_2, \beta_1, \gamma_0) =(1, \beta_1,1), \quad \Big(\alpha_1, \alpha_0;  \beta_2, \beta_0;  \gamma_2, \gamma_1 \Big)
 \mapsto
  \Big( \frac{1}{\nu} \alpha_1, \frac{1}{\nu^2} \alpha_0; \nu \beta_2, \frac{1}{\nu} \beta_0; \nu^2 \gamma_2, \nu \gamma_1\Big).
\end{array}}
\end{equation}
\begin{remark}
One can use the scaling in Equation~(\ref{scalingWEq}) to fix $\alpha_2=\gamma_0=1$. The curves $ \hat{\mathcal{C}}_{\alpha}$ and $ \hat{\mathcal{C}}_{\gamma}$ have rational points and are rationally isomorphic to the elliptic curves
\begin{equation}
\label{curves}
\begin{split}
  \hat{\mathcal{E}}_\alpha: \quad & Y^2  =  X \Big(X^2 -4  \alpha_1 X + 4 (\alpha_1^2-\alpha_0) \Big),\\
  \hat{\mathcal{E}}_\gamma: \quad & Y^2   =  X \Big(X^2 -4  \gamma_1 X + 4(\gamma_1^2 -\gamma_2) \Big).
\end{split}
\end{equation}
\end{remark}
\subsection{Recovering the moduli of the CHL string}
The holomorphic two-form on  the K3 surface $ \tilde{\mathcal{Y}}$ in Equation~(\ref{eqns_Kummer_up2b}) is given by
\begin{equation}
 \Big(s_1 ds_0-s_0 ds_1\Big) \wedge \frac{\tilde{Z} \, d\tilde{X} - \tilde{X} \, d\tilde{Z}}{\tilde{Y}} .
\end{equation}
The K3 surface $\tilde{\mathcal{Y}}$ admits the Nikulin involution in Equation~(\ref{Nikulin_base}) and a second Nikulin involution $\imath_{\tilde{\mathcal{Y}}}$ -- acting by fiberwise translation by two-torsion -- given by
\begin{equation}
  \imath_{\tilde{\mathcal{Y}}}: \;  \Big([s_0:s_1],[\tilde{X}: \tilde{Y}: \tilde{Z}]\Big)  \mapsto  \Big([s_0: s_1],[ \alpha\tilde{z}: -\alpha\gamma\tilde{y}: \gamma\tilde{x}]\Big).
\end{equation}
As shown in Diagram~(\ref{pic:VGS}), the K3 surfaces $\tilde{\mathcal{X}}$ and $\tilde{\mathcal{Y}}$ are obtained by fiberwise two-isogeny from each other. The  map $\tilde{\Phi}: \tilde{Y} \dasharrow \tilde{X}$ is the rational degree-two map associated with the quotient $\tilde{\mathcal{Y}}/\{\mathbb{I}, \imath_{\tilde{\mathcal{Y}}}\}$. Moreover, it is obvious that $\tilde{\mathcal{Y}}$ also admits the anti-symplectic involution
\begin{equation}
\begin{array}{lrcl}
  \delta_{\tilde{\mathcal{Y}}}:&  \Big([s_0:s_1],[\tilde{X}: \tilde{Y}: \tilde{Z}]\Big) & \mapsto & \Big([s_0:-s_1],[\tilde{X}: \tilde{Y}: \tilde{Z}]\Big).
\end{array}
\end{equation}
Similar statements apply to $\tilde{\mathcal{X}}$, and we will denote the anti-symplectic involution on $\tilde{\mathcal{X}}$ by $\delta_{\tilde{\mathcal{X}}}$. In particular, we have $\delta_{\tilde{\mathcal{X}}} \circ \tilde{\Phi} = \tilde{\Phi} \circ \delta_{\tilde{\mathcal{Y}}}$. We have the following:
\begin{lemma}
\label{fullMW}
The minimal resolution of $\tilde{\mathcal{X}} / \{ \mathbb{I}, \delta_{\tilde{\mathcal{X}}} \}$ is the rational elliptic surface $\mathcal{J}$ with singular fibers $4 I_2 + 4 I_1$ and a Mordell-Weil group $\mathbb{Z}/2\mathbb{Z} \oplus D_4^*$ where $D_4^*$ is the dual of the root lattice of type $D_4$.  It follows that the Mordell-Weil group of the K3 surface $\tilde{\mathcal{X}}$ in Equation~(\ref{eqns_Kummer_up2}) is $\operatorname{MW}(\pi_{\tilde{\mathcal{X}}}) =\mathbb{Z}/2\mathbb{Z} \oplus D_4^*(-2)$.
\end{lemma}
\begin{proof}
The rational elliptic surface $\mathcal{J}$ is the minimal resolution of $\tilde{\mathcal{X}} / \{ \mathbb{I}, \delta_{\tilde{\mathcal{X}}} \}$. An elliptic fibration $\pi_{\mathcal{J}}: \mathcal{J} \to \mathbb{P}^1$ with section is obtained using coordinates $[x:y:z]\in \mathbb{P}^2$ for the fiber and $[t_0:t_1]\in \mathbb{P}^1_{[t_0:t_1]}$ for the base from the equation
\begin{equation}
\label{RES_J}
  \mathcal{J}_{[t_0:t_1]}: \quad y^2 z  = x \Big(x^2 + 2 \beta  x z+  \alpha\gamma z^2\Big), \quad  \Delta=4\alpha^2\gamma^2(\beta^2-\alpha\gamma) .
 \end{equation}
It follows that the singular fibers are $4 I_2 + 4 I_1$ and the Mordell-Weil group sections contains $\mathbb{Z}/2\mathbb{Z}$. In the classification or rational elliptic surfaces in \cite{MR1104782} this is surface labeled {\tt (13)} with the Mordell-Weil group $\operatorname{MW}(\pi_{\tilde{\mathcal{J}}}) = \operatorname{MW}(\pi_{\tilde{\mathcal{J}}})_{\mathrm{tor}} \oplus L^*$ where the torsion subgroup is $\mathbb{Z}/2\mathbb{Z}$ and the narrow  Mordell-Weil lattice $L$ is the root lattice $D_4$. In the well-known formula for the height pairing on an elliptically fibered surface with section all summands -- the holomorphic Euler characteristic, the intersection number, the number of singular fibers -- double when pulling the Weierstrass Equation~(\ref{RES_J}) back via the map $f: \mathbb{P}^1 \to \mathbb{P}^1$ given by $[s_0:s_1] \mapsto [t_0:t_1] =[s^2_0:s^2_1]$ to obtain Equation~(\ref{eqns_Kummer_up2}). It follows that $\operatorname{MW}(\pi_{\tilde{\mathcal{X}}}) =\mathbb{Z}/2\mathbb{Z} \oplus D_4^*(-2)$ where $D_4^*(-2)$ is the dual of the root lattice of type $D_4$ whose intersection form has been rescaled by $(-2)$.
\end{proof}
We have the following:
\begin{lemma}
The K3 surface $\tilde{\mathcal{Y}}$ admits the fixed-point-free involution $\epsilon = \delta_{\tilde{\mathcal{Y}}} \circ \imath_{\tilde{\mathcal{Y}}}$.
\end{lemma}
\begin{proof}
As the two-torsion section does not specialize to the zero section on the two fibers fixed by the deck transformation -- since we assumed that $s_0=0$ and $s_1=0$ were generic elliptic fibers -- the translation by the two-torsion section composed with the deck transformation is fixed-point-free.
\end{proof}
From the pair $(\tilde{\mathcal{Y}},\epsilon)$ we can construct an Enriques surface $\mathcal{S}=\tilde{\mathcal{Y}} / \{ \mathbb{I}, \epsilon \}$, i.e., a smooth projective minimal regular algebraic surface of Kodaira dimension zero with Betti number $b_2 = 10$, as the quotient of the K3 surface by the the fixed-point-free involution. We denote the associated \'etale double cover by $p: \tilde{\mathcal{Y}} \to \mathcal{S}$, and the canonical classes on $\mathcal{S}$ and $\tilde{\mathcal{Y}}$ satisfy $K_{\tilde{\mathcal{Y}}} = p^*(K_{\mathcal{S}}) = \mathcal{O}_{\tilde{\mathcal{Y}}}$ where $K_{\mathcal{S}} \not =0$ and $K_{\mathcal{S}}$ is two-torsion $2 K_{\mathcal{S}}=0$.  The Enriques surface $\mathcal{S}$ admits an elliptic fibration, albeit not with a section. However, we can again pass to the relative Jacobian elliptic fibration of $\pi_{\mathcal{S}}$ to guarantees the existence of a section while preserving many properties, for instance the Picard number. In particular, the relative Jacobian elliptic fibration is a rational elliptic fibration. A classical construction of Kodaira, the so called logarithmic transform, then allows us to reconstruct $\mathcal{S}$. We have the following:
\begin{proposition}
The Enriques surface $\mathcal{S}=\tilde{\mathcal{Y}}(\check{\alpha}, \check{\beta}, \check{\gamma}) / \{ \mathbb{I}, \epsilon \}$ admits a fibration $\pi_{\mathcal{S}}: \mathcal{S} \to \mathbb{P}^1$ whose general fiber is a curve of arithmetic genus one. For $t_0t_1\not =0$ the fibration is analytically isomorphic to the elliptic fibration $\pi_{\mathcal{J}}: \mathcal{J} \to \mathbb{P}^1$ with section on the rational elliptic surface $\mathcal{J}$ in Equation~(\ref{RES_J}), and it has two multiple fibers of the form $2F$ over $[t_0:t_1]=[0:1]$ and  $[t_0:t_1]=[1:0] \in \mathbb{P}^1$ with fibers
\begin{equation}
\begin{array}{ll}
  F_{[0:1]}: & y^2 z  =  x \Big(x^2 + 2  \check{\beta}_0 xz + \check{\alpha}_0\check{\gamma}_0 z^2\Big),\\[0.5em]
  F_{[1:0]}: & y^2 z  =  x \Big(x^2 + 2  \check{\beta}_2 xz + \check{\alpha}_2\check{\gamma}_2 z^2\Big).
\end{array}
\end{equation}
In particular, the relative Jacobian elliptic fibration of $\pi_{\mathcal{S}}$ is the elliptic fibration $\pi_{\mathcal{J}}: \mathcal{J} \to \mathbb{P}^1_{[t_0:t_1]}$ with section.
\end{proposition}
\begin{proof}
It is well known \cite{MR3586505} that a fibration $\pi_{\tilde{\mathcal{X}}}: \tilde{\mathcal{X}} \to \mathbb{P}^1_{[s_0:s_1]}$ induces a fibration $\pi_{\mathcal{S}}: \mathcal{S} \to \mathbb{P}^1_{[t_0:t_1]}$ with two double-fibers and a general fiber that is a curve of arithmetic genus one. The elliptic fibration $\pi_{\mathcal{S}}: \mathcal{S} \to \mathbb{P}^1_{[t_0:t_1]}$ must have the two multiple fibers over the ramification points of the map $f: [s_0:s_1] \mapsto [t_0:t_1] =[s^2_0:s^2_1]$ invariant under the involution $\delta_B$. This can be seen as follows: away from the branching points $[t_0:t_1]=[0:1]$ and $[t_0:t_1]=[1:0]$, the elliptic fibration is given by $\pi_{\mathcal{J}}$ since $\mathcal{J}$ was obtained from $\tilde{\mathcal{X}} / \{ \mathbb{I}, \delta_{\tilde{\mathcal{Y}}} \}$, $\tilde{\mathcal{X}}$ from $\tilde{\mathcal{Y}}/\{\mathbb{I}, \imath_{\tilde{\mathcal{Y}}}\}$, and $\epsilon = \delta_{\tilde{\mathcal{Y}}} \circ \imath_{\tilde{\mathcal{Y}}}$. That is, Equation~(\ref{RES_J}) determines the genus-one fibration $\pi_{\mathcal{S}}$ outside of two small discs around $[t_0:t_1]=[0:1]$ and $[t_0:t_1]=[1:0]$.
\par We now construct the fibration $\pi_{\mathcal{S}}$ inside an excised small analytic disc centered at one of the branching points, say $[t_0,t_1]=[0:1]$. The fibration inside the other disc is then constructed in a similar way. We use the affine coordinate $s=s_0/s_1$ and consider a small analytic disc $D \subset \mathbb{P}^1_{[s_0:s_1]}$ centered at $s=0$ with image $D'=f(D)\subset\mathbb{P}^1_{[t_0:t_1]}$. Since the fiber over $s=0$ is regular, the fibration $\pi_{\tilde{\mathcal{Y}}}|_D:\tilde{\mathcal{Y}} \to D$ is topologically trivial, and $\tilde{\mathcal{Y}}\mid_D$ is isomorphic to $\hat{\mathcal{E}} \times D$ where $\hat{\mathcal{E}}\cong \mathbb{C}/(\mathbb{Z}\oplus \mathbb{Z}\pmb{\tau} )$. Moreover, the projection $\pi_{\tilde{\mathcal{Y}}}|_D$ is the projection $\big(z, s\big) \to s$ for $z\in \hat{\mathcal{E}}$ and $s\in D$. The Enriques involution $\epsilon$ acts on $\hat{\mathcal{E}} \times D$ as an automorphism of order two by $\big(z,s\big) \mapsto \big(z+\frac{1}{2},-s\big)$. The quotient $\mathcal{Q}=\hat{\mathcal{E}} \times D/\epsilon$ is a fibration over the disc $D'$ with projection $\big[(z,s)\big] \mapsto t=s^2$ and a fiber of multiplicity two over $t=0$. The two fibrations are analytically isomorphic over the punctured discs $D - \{0\}$ and $D' - \{0\}$ by mapping $\big(z,s\big)$ to $\big(z',t\big)=\big(z-\log{(s)}/(2\pi i),s\big)$. The central fiber is of the form $2F$ where $F \cong \mathbb{C}/(\mathbb{Z}/2\mathbb{Z}\oplus \mathbb{Z} \pmb{\tau})$. After rescaling, $F$ becomes the two-isogeneous curve $\mathcal{E}\cong \mathbb{C}/(\mathbb{Z}\oplus 2\mathbb{Z} \pmb{\tau})$.
\end{proof}
We already proved that the relative Jacobian fibration of the genus-one fibration on $\mathbf{X}=\tilde{\mathcal{X}}(\alpha, \beta, \gamma)$ with two singular fibers of type $I_0^*$ is the elliptic fibration with section $\pi_{\mathcal{Y}} : \mathcal{Y} \to \mathbb{P}^1$. In turn, this elliptic fibration is inherited from the elliptic fibration with section on the K3 surface $\mathbf{Y}:=\tilde{\mathcal{Y}}(\check{\alpha}, \check{\beta}, \check{\gamma})$, namely as quotient by the Nikulin involution in Equation~(\ref{Nikulin_base}). The different sets of moduli, i.e., $(\alpha, \beta, \gamma)$ and $(\check{\alpha}, \check{\beta}, \check{\gamma})$, are related by Equations~(\ref{dual_6L}) and appear due to Lemma~\ref{lem:EquivFibs}. The K3 surface $\mathbf{Y}$ admits a fixed-point-free involution, and the resulting Enriques surface is equivalently described by the rational elliptic surface $\mathcal{J}$ in Equation~(\ref{RES_J}) together with the two double-fibers: they are two elliptic curves with marked two-torsion points given by
\begin{equation}
\label{curves_dual}
\begin{split}
  \mathcal{E}_{[1:0]}: \quad & y^2  =  x \Big(x^2 + 2  \check{\beta}_2 x + \check{\alpha}_2\check{\gamma}_2 \Big),\\
  \mathcal{E}_{[0:1]}: \quad & y^2   =  x \Big(x^2 + 2  \check{\beta}_0 x + \check{\alpha}_0\check{\gamma}_0 \Big).
\end{split}
\end{equation}
By construction, scaling according to Equation~(\ref{scalingWEq_dual}) acts on the rational elliptic surface $\mathcal{J}$ in Equation~(\ref{RES_J}) in the same way scaling is acts on the K3 surfaces, by holomorphic isomorphism preserving fibers. Thus, scaling according to Equation~(\ref{scalingWEq_dual}) acts as holomorphic isomorphism on the elliptic curves $\mathcal{E}_{[1:0]}$ and $\mathcal{E}_{[0:1]}$. The corresponding two-isogeneous elliptic curves are
\begin{equation}
\begin{split}
  \hat{\mathcal{E}}_{[1:0]}: \quad & Y^2  =  X \Big(X^2 -4  \check{\beta}_2 X + 4 (\check{\beta}^2_2-\check{\alpha}_2\check{\gamma}_2) \Big),\\
  \hat{\mathcal{E}}_{[0:1]}: \quad & Y^2   =  X \Big(X^2 -4  \check{\beta}_0 X + 4(\check{\beta}^2_0 -\check{\alpha}_0\check{\gamma}_0) \Big).
\end{split}
\end{equation}
Using the relation between the moduli and the dual moduli in Equations~(\ref{dual_6L}) it follows that the two elliptic curves $\hat{\mathcal{E}}_{[1:0]}$ and $\hat{\mathcal{E}}_{[0:1]}$ are isomorphic to the genus-one curves $\hat{\mathcal{C}}_{\alpha}$ and $\hat{\mathcal{C}}_{\gamma}$ in Equation~(\ref{curve1}) and (\ref{curve2}) and $\hat{\mathcal{E}}_\alpha$ and $\hat{\mathcal{E}}_\gamma$ in Equations~(\ref{curves}) in the sense of Lemma~\ref{lemma_reduction} and Proposition~\ref{Jac_prop}. Therefore, we will denote the elliptic curves in Equations~(\ref{curves_dual}) by  $\mathcal{E}_\alpha$ and $\mathcal{E}_\gamma$.
\par It was proved in \cite{MR1675162} that an $E_8$ bundle $\mathbf{W} \to \mathcal{E}$ over an elliptic curve $\mathcal{E}$ is isomorphic to a rational elliptic surface $\mathcal{J}$ with section together with an embedding of $\mathcal{E}$ into $\mathcal{J}$ as anti-canonical curve, in particular as fiber over infinity.  Thus, the obvious embeddings of the elliptic curves $\mathcal{E}_\alpha$ and $\mathcal{E}_\gamma$ as anti-canonical curves into $\mathcal{J}$ determine two $E_8$ bundles $\mathbf{W}^{\mathcal{J}}_\alpha \to \mathcal{E}_\alpha$ and $\mathbf{W}^{\mathcal{J}}_\gamma \to \mathcal{E}_\gamma$.  
We have therefore proved the following:
\begin{theorem}
\label{prop:duality}
The defining data of the CHL string (as explained in Section~\ref{sec:CHL}) dual to the eight-dimensional F-theory compactifications $(\mathbf{X}, \pi_{\mathbf{X}}, \alpha)$ with Picard number $14$
is ($a'$) the elliptic curve
\begin{equation}
  \mathcal{E}_\alpha: \quad  y^2  =  x \Big(x^2 + 2  \alpha_1 x + \alpha_2\alpha_0 \Big),\\
\end{equation}
($b'$) the $\mathbb{Z}/2\mathbb{Z}$ subgroup  $\{\sigma, \tau\}$ of $\mathcal{E}_\alpha$ inducing the involution $\imath_{\mathcal{E}_\alpha}$ in Equation~(\ref{involution_fiber}) -- where $\tau$ is the two-torsion point $(x,y)=(0,0)$ and $\sigma$ the point at infinity -- and ($c'$) the principal $E_8$ bundle $\mathbf{W}^{\mathcal{J}}_\alpha \to \mathcal{E}_\alpha$ obtained from the embedding $\mathcal{E}_\alpha \hookrightarrow \mathcal{J}$ in Equation~(\ref{RES_J}).
\end{theorem}
\begin{remark}
We use Equation~(\ref{scalingWEq_2}) to fix $\alpha_2=\gamma_0=1$ and use the residual scaling to fix $\alpha_1$ to a non-vanishing value $\alpha_1\not= 0$ since $\mathcal{E}_\alpha$ is smooth. For $\alpha_2=\alpha_1=\gamma_0=1$, the moduli of the CHL string dual to the eight-dimensional F-theory compactifications $(\mathbf{X}, \pi_{\mathbf{X}}, \alpha)$ are given by $(\beta_2, \beta_1, \beta_0; \gamma_2, \gamma_1)$ determining the $E_8$ bundle $\mathbf{W}^{\mathcal{J}}_\alpha \to \mathcal{E}_\alpha$ and
\begin{equation}
 j(\mathcal{E}_\alpha) = \frac{1}{27} \frac{(2\alpha_0-4)^3}{\alpha_0^2(\alpha_0-1)} \;.
\end{equation}
\end{remark}
\begin{remark}
Choosing as CHL data the second double-fiber $\mathcal{E}_\gamma$ of the Enriques surface and the principal $E_8$ bundle $\mathbf{W}^{\mathcal{J}}_\gamma \to \mathcal{E}_\gamma$ corresponds to the eight-dimensional F-theory compactifications $(\mathbf{X}, \pi_{\mathbf{X}}, \gamma)$ with Picard number $14$.
\end{remark}
\bibliography{CM17} 
\bibliographystyle{amsplain}
\end{document}